\documentclass[a4paper,12pt]{amsart}
\calclayout
\usepackage{amssymb}
\usepackage{amsmath}
\usepackage{amsfonts}
\usepackage[all]{xy}
\usepackage{mathabx}
\usepackage[dvipdfm,CJKbookmarks,bookmarks=true,colorlinks=false]{hyperref}

\newtheorem{thm}{Theorem}[section]
\newtheorem{cor}[thm]{Corollary}

\newtheorem{prop}[thm]{Proposition}
\theoremstyle{definition}

\theoremstyle{remark}
\newtheorem{rem}[thm]{Remark}
\numberwithin{equation}{section}

\begin{document}

\title{A Kirchberg type tensor theorem for operator systems}

\author{Kyung Hoon Han}

\address{Department of Mathematics, The University of Suwon, Gyeonggi-do 445-743, Korea}

\email{kyunghoon.han@gmail.com}

\subjclass[2000]{46L06, 46L07, 47L07}

\keywords{operator system, tensor product, quotient, Kirchberg's theorem}

\thanks{This work was supported by the National Research Foundation of Korea Grant funded by the Korean Government (NRF-2012R1A1A1012190)}

\date{}

\dedicatory{}

\commby{}


\begin{abstract}
We construct operator systems $\mathfrak C_I$ that are universal in the sense that all operator systems can be realized as their quotients. They satisfy the operator system lifting property. Without relying on the theorem by Kirchberg, we prove the Kirchberg type tensor theorem $$\mathfrak C_I \otimes_{\min} B(H) = \mathfrak C_I \otimes_{\max} B(H).$$ Combining this with a result of Kavruk, we give a new operator system theoretic proof of Kirchberg's theorem and show that Kirchberg's conjecture is equivalent to its operator system analogue $$\mathfrak C_I \otimes_{\min} \mathfrak C_I =\mathfrak C_I \otimes_{\rm c} \mathfrak C_I.$$

It is natural to ask whether the universal operator systems $\mathfrak C_I$ are projective objects in the category of operator systems. We show that an operator system from which all unital completely positive maps into operator system quotients can be lifted is necessarily one-dimensional. Moreover, a finite dimensional operator system satisfying a perturbed lifting property can be represented as the direct sum of matrix algebras. We give an operator system theoretic approach to the Effros-Haagerup lifting theorem.
\end{abstract}

\maketitle

\section{Introduction}
Every Banach space can be realized as a quotient of $\ell_1(I)$ for a suitable choice of index set $I$. Moreover, every linear map $\varphi : \ell_1(I) \to E \slash F$ lifts to $\tilde{\varphi} : \ell_1(I) \to E$ with $\|\tilde{\varphi}\|<(1+\varepsilon) \|\varphi\|$. On noncommutative sides, $\oplus_1 T_{n_i}$ (respectively $C^*(\mathbb F)$) plays such a role in the category of operator spaces (respectively $C^*$-algebras). The purpose of this paper is to find operator systems that play such a role in the category of operator systems.

We construct operator systems $\mathfrak C_I$ that are universal in the sense that all operator systems can be realized as their quotients. The method of construction is motivated by \cite[Proposition 3.1]{Bl} and the coproduct of operator systems \cite{F,KL}. The index set $I$ is chosen to be sufficiently large that we can index the set $\mathcal S_{\|\cdot\| \le 1}^+$ of positive contractive elements in an operator system $\mathcal S$. The operator system $\mathfrak C_I$ is realized as the infinite coproduct of $\{ M_k \oplus M_k \}_{k \in \mathbb N}$ admitting copies of $M_k \oplus M_k$ up to the cardinality of $I$.

We prove that the operator systems $\mathfrak C_I$ satisfy the operator system lifting property: for any unital $C^*$-algebra $\mathcal A$ with its closed ideal $\mathcal I$ and the quotient map $\pi : \mathcal A \to \mathcal A \slash \mathcal I$, every unital completely positive map $\varphi : \mathfrak C_I \to \mathcal A \slash \mathcal I$ lifts to a unital completely positive map $\tilde{\varphi} : \mathfrak C_I \to \mathcal A$. It is helpful to picture the situation using the commutative diagram $$\xymatrix{ & \mathcal A \ar[d]^{\pi} \\ \mathfrak C_I \ar@{-->}[ur]^{\tilde{\varphi}} \ar[r]^{\varphi}& \mathcal A \slash \mathcal I.}$$

For a free group $\mathbb F$ and a Hilbert space $H$, Kirchberg \cite[Corollary 1.2]{Ki} proved that $$C^*(\mathbb F) \otimes_{\min} B(H) = C^*(\mathbb F) \otimes_{\max} B(H).$$  The proof was later simplified in \cite{P2} and \cite{FP} using operator space theory and operator system theory, respectively. Kirchberg's theorem is striking if we recall that $C^*(\mathbb F)$ and $B(H)$ are universal objects in the $C^*$-algebra category: every $C^*$-algebra is a $C^*$-quotient of $C^*(\mathbb F)$ and a $C^*$-subalgebra  of $B(H)$ for suitable choices of $\mathbb F$ and $H$.

For suitable choices of $I$ and $H$, every operator system is a subsystem of $B(H)$ by the Choi-Effros theorem \cite{CE} and is a quotient of $\mathfrak C_I$ which is proved in Section 3. We will prove a Kirchberg type tensor theorem $$\mathfrak C_I \otimes_{\min} B(H) = \mathfrak C_I \otimes_{\max} B(H)$$ in Section 4. The proof is independent of Kirchberg's theorem. Combining this with Kavruk's idea \cite{Ka2}, we give a new operator system theoretic proof of Kirchberg's theorem.

We also prove that the operator system analogue $$\mathfrak C_I \otimes_{\min} \mathfrak C_I =\mathfrak C_I \otimes_c \mathfrak C_I$$ of Kirchberg's conjecture $$C^*(\mathbb F) \otimes_{\min} C^*(\mathbb F) = C^*(\mathbb F) \otimes_{\max} C^*(\mathbb F)$$ is equivalent to Kirchberg's conjecture itself.

In the final section, we consider several lifting problems of completely positive maps. It is natural to ask whether the universal operator system $\mathfrak C_I$ is a projective object in the category of operator systems. In other words, for any operator system $\mathcal S$ and its kernel $\mathcal J$, does every unital completely positive map $\varphi : \mathfrak C_I \to \mathcal S \slash \mathcal J$ lift to a unital completely positive map $\tilde{\varphi} : \mathfrak C_I \to \mathcal S$? The answer is negative in an extreme manner. An operator system satisfying such a lifting property is necessarily one-dimensional. This is essentially due to Archimedeanization of quotients \cite{PT}. Even though some perturbation is allowed, there is also rigidity: for a finite dimensional operator system $E$ and a faithful state $\omega$, the following are equivalent:
\begin{enumerate}
\item[(i)] if $\varepsilon>0$ and $\varphi : E \to \mathcal S \slash \mathcal J$ is a completely positive map for an operator system $\mathcal S$ and its kernel $\mathcal J$, then there exists a self-adjoint lifting $\tilde{\varphi} : E \to \mathcal S$ of $\varphi$ such that $\tilde{\varphi}+\varepsilon \omega 1_{\mathcal S}$ is completely positive;
\item[(ii)] $E$ is unitally completely order isomorphic to the direct sum of matrix algebras.
\end{enumerate}
In order to prove it, we give a characterization of nuclearity via the projectivity and the minimal tensor product: an operator system $\mathcal S$ is nuclear if and only if $${\rm id}_{\mathcal S} \otimes \Phi : \mathcal S \otimes_{\min} \mathcal T_1 \to \mathcal S \otimes_{\min} \mathcal T_2$$ is a quotient map for any quotient map $\Phi : \mathcal T_1 \to \mathcal T_2$.

Finally, we present an operator system theoretic approach to the Effros-Haagerup lifting theorem \cite{EH}.

\section{Preliminaries}

Let $\mathcal S$ and $\mathcal T$ be operator systems. Following
\cite{KPTT1}, an {\em operator system structure} on $\mathcal S \otimes
\mathcal T$ is defined as a family of cones $M_n (\mathcal S \otimes_\tau
\mathcal T)^+$ satisfying:
\begin{enumerate}
\item[(T1)] $(\mathcal S \otimes \mathcal T, \{ M_n (\mathcal S \otimes_\tau \mathcal
T)^+ \}_{n=1}^\infty, 1_{\mathcal S} \otimes1_{\mathcal T})$ is an operator
system denoted by $\mathcal S \otimes_\tau \mathcal T$, \item[(T2)] $M_m(\mathcal
S)^+ \otimes M_n(\mathcal T)^+ \subset M_{mn} (\mathcal S \otimes_\tau \mathcal
T)^+$ for all $m, n \in \mathbb N$, and \item[(T3)] if $\varphi :
\mathcal S \to M_m$ and $\psi : \mathcal T \to M_n$ are unital completely
positive maps, then $\varphi \otimes \psi : \mathcal S \otimes_\tau \mathcal
T \to M_{mn}$ is a unital completely positive map.
\end{enumerate}
By an {\em operator system tensor product,} we mean a mapping
$\tau : \mathcal O \times \mathcal O \to \mathcal O$, such that for every pair of
operator systems $\mathcal S$ and $\mathcal T$, we have that $\tau (\mathcal S, \mathcal T)$ is an
operator system structure on $\mathcal S \otimes \mathcal T$, and denote it by $\mathcal S
\otimes_\tau \mathcal T$. We call an operator system tensor product
$\tau$ {\em functorial,} if the following property is satisfied:
\begin{enumerate}
\item[(T4)] for any operator systems $\mathcal S_1, \mathcal S_2, \mathcal T_1,
\mathcal T_2$ and unital completely positive maps $\varphi : \mathcal S_1
\to \mathcal T_1, \psi : \mathcal S_2 \to \mathcal T_2$, the map $\varphi \otimes
\psi : \mathcal S_1 \otimes \mathcal S_2 \to \mathcal T_1 \otimes \mathcal T_2$ is
unital completely positive.
\end{enumerate}

Given a linear mapping $\varphi : V \to W$ between vector spaces, its $n$-th amplification $\varphi_n : M_n(V) \to M_n(W)$ is defined as $\varphi_n([x_{i,j}])=[\varphi(x_{i,j})]$. For operator systems $\mathcal S$ and $\mathcal T$, we put
$$\begin{aligned} M_n(\mathcal S \otimes_{\min} \mathcal T)^+ = \{ X & \in
M_n(\mathcal S \otimes \mathcal T) :  (\varphi \otimes \psi)_n(X) \in
M_{nkl}^+ \text{ for all unital}\\ & \text{completely positive maps } \varphi : \mathcal S \to M_k, \psi : \mathcal T \to M_l \}. \end{aligned}$$  Then the family $\{ M_n(\mathcal S
\otimes_{\min} \mathcal T)^+ \}_{n=1}^\infty$ is an operator system
structure on $\mathcal S \otimes \mathcal T$.
Moreover, if we let $\iota_{\mathcal S} : \mathcal S \to \mathcal B(\mathcal H)$
and $\iota_{\mathcal T} : \mathcal T \to \mathcal B(\mathcal K)$ be any unital complete
order embeddings, then this is the operator system structure on
$\mathcal S \otimes \mathcal T$  arising from the embedding
$\iota_{\mathcal S} \otimes \iota_{\mathcal T} : \mathcal S \otimes \mathcal T \to \mathcal
B(\mathcal H \otimes \mathcal K)$ \cite[Theorem 4.4]{KPTT1}. We call the operator
system $(\mathcal S \otimes \mathcal T, \{ M_n(\mathcal S \otimes_{\min} \mathcal T)
\}_{n=1}^\infty, 1_{\mathcal S} \otimes 1_{\mathcal T})$ the {\em minimal}
tensor product of $\mathcal S$ and $\mathcal T$ and denote it by $\mathcal S
\otimes_{\min} \mathcal T$.

The mapping $\min : \mathcal O \times \mathcal O \to \mathcal O$ sending $(\mathcal S,
\mathcal T)$ to $\mathcal S \otimes_{\min} \mathcal T$ is an injective,
associative, symmetric and functorial operator system tensor
product. The positive cone of the minimal tensor product is the
largest among all possible positive cones of operator system
tensor products at each matrix level \cite[Theorem 4.6]{KPTT1}. The operator system minimal tensor product $\mathcal A \otimes_{\min} \mathcal B$ of unital $C^*$-algebras $\mathcal
A$ and $\mathcal B$ is a dense subsystem of $C^*$-minimal tensor product $\mathcal A \otimes_{\rm C^*\min} \mathcal B$ \cite[Corollary 4.10]{KPTT1}.

For operator systems $\mathcal S$ and $\mathcal T$, we put
$$D_n^{\max}(\mathcal S, \mathcal T) = \{ \alpha(P \otimes Q)
\alpha^* : P \in M_k(\mathcal S)^+, Q \in M_l(\mathcal T)^+, \alpha \in
M_{n,kl},\ k,l \in \mathbb N \}.$$ This is a matrix ordering on
$\mathcal S \otimes \mathcal T$ with order unit $1_{\mathcal S} \otimes 1_{\mathcal
T}$. Let $\{ M_n(\mathcal S \otimes_{\max} \mathcal T)^+ \}_{n=1}^\infty$ be
the Archimedeanization of the matrix ordering $\{ D_n^{\max}(\mathcal
S, \mathcal T) \}_{n=1}^\infty$. Then it can be written as
$$M_n(\mathcal S \otimes_{\max} \mathcal T)^+ = \{ X \in M_n(\mathcal S
\otimes \mathcal T) : \forall \varepsilon>0, X+\varepsilon I_n \otimes
1_{\mathcal S} \otimes 1_{\mathcal T} \in D_n^{\max}(\mathcal S, \mathcal T) \}.$$ We
call the operator system $(\mathcal S \otimes \mathcal T, \{ M_n(\mathcal S
\otimes_{\max} \mathcal T)^+ \}_{n=1}^\infty, 1_{\mathcal S} \otimes 1_{\mathcal
T})$ the {\em maximal} operator system tensor product of $\mathcal S$
and $\mathcal T$ and denote it by $\mathcal S \otimes_{\max} \mathcal T$.

The mapping $\max : \mathcal O \times \mathcal O \to \mathcal O$ sending $(\mathcal S,
\mathcal T)$ to $\mathcal S \otimes_{\max} \mathcal T$ is an associative,
symmetric and functorial operator system tensor product. The
positive cone of the maximal tensor product is the smallest among
all possible positive cones of operator system tensor products at each matrix level
\cite[Theorem 5.5]{KPTT1}. The operator system maximal tensor product $\mathcal A \otimes_{\max} \mathcal B$ of unital $C^*$-algebras $\mathcal A$ and $\mathcal B$ is a dense subsystem of $C^*$-maximal tensor product $\mathcal A \otimes_{\rm C^*\max} \mathcal B$ \cite[Theorem 5.12]{KPTT1}.

For completely positive maps $\varphi : \mathcal S \to B(H)$ and $\psi : \mathcal T \to B(H)$, let $\varphi \cdot \psi : \mathcal S \otimes \mathcal T \to B(H)$ be the map given on simple tensors by $(\varphi \cdot \psi)(x \otimes y) = \varphi (x) \psi (y)$. We put
$$
\begin{aligned}
M_n(\mathcal S \otimes_{\rm c} \mathcal T)^+ = \{ X \in M_n (\mathcal S \otimes & \mathcal T) : (\varphi \cdot \psi)_n (X) \ge 0 \text{ for all completely positive maps} \\ & \varphi : \mathcal S \to B(H), \psi : \mathcal T \to B(H) \text{ with commuting ranges}\}.
\end{aligned}
$$
Then the family $\{ M_n(\mathcal S \otimes_{\rm c} \mathcal T)^+ \}_{n=1}^\infty$ is an operator system structure on $\mathcal S \otimes \mathcal T$. We call the operator
system $(\mathcal S \otimes \mathcal T, \{ M_n(\mathcal S \otimes_{\rm c} \mathcal T) \}_{n=1}^\infty, 1_{\mathcal S} \otimes 1_{\mathcal T})$ the {\em commuting}
tensor product of $\mathcal S$ and $\mathcal T$ and denote it by $\mathcal S
\otimes_{\rm c} \mathcal T$. If $\mathcal A$ is a unital $C^*$-algebra and $\mathcal S$ is an operator system, then we have $$\mathcal A \otimes_{\rm c} \mathcal S = \mathcal A \otimes_{\max} \mathcal S$$ \cite[Theorem 6.7]{KPTT1}. Hence, the maximal tensor product and the commuting tensor product are two different means of extending the $C^*$-maximal tensor product from the category of $C^*$-algebras to operator systems.

For an inclusion $\mathcal S \subset B(H)$, we let $\mathcal S \otimes_{\rm el} \mathcal T$ be the operator system with underlying space $\mathcal S \otimes \mathcal T$ whose matrix ordering is induced by the inclusion $\mathcal S \otimes \mathcal T \subset B(H) \otimes_{\max} \mathcal T$. We call the operator system $\mathcal S \otimes_{\rm el} \mathcal T$ the {\em enveloping left} operator system tensor product of $\mathcal S$ and $\mathcal T$. The mapping ${\rm el} : \mathcal O \times \mathcal O \to \mathcal O$ sending $(\mathcal S, \mathcal T)$ to $\mathcal S \otimes_{\rm el} \mathcal T$ is a left injective functorial operator system tensor product. Here, $\mathcal S \otimes_{\rm el} \mathcal T$ is independent of the choice of any injective operator system containing $\mathcal S$ instead of $B(H)$.

Given an operator system $\mathcal S$, we call $\mathcal J \subset \mathcal S$ the kernel, provided that it is the kernel of a unital completely positive map from $\mathcal S$ to another operator system. If we define a family of positive cones $M_n(\mathcal S \slash \mathcal J)^+$ on $M_n(\mathcal S \slash \mathcal J)$ as $$M_n(\mathcal S \slash \mathcal J)^+ := \{ [x_{i,j}+\mathcal J]_{i,j} : \forall \varepsilon >0, \exists k_{i,j} \in \mathcal J,\ \varepsilon I_n \otimes 1_{\mathcal S} + [x_{i,j}+k_{i,j}]_{i,j} \in M_n(\mathcal S)^+ \},$$ then $(\mathcal S \slash \mathcal J, \{ M_n(\mathcal S \slash \mathcal J)^+ \}_{n=1}^\infty, 1_{\mathcal S \slash \mathcal J}) $ satisfies all the conditions of an operator system \cite[Proposition 3.4]{KPTT2}. We call this the quotient operator system. With this definition, the first isomorphism theorem can be proved: If $\varphi : \mathcal S \to \mathcal T$ is a unital completely positive map with $\mathcal J \subset \ker \varphi$, then the map $\widetilde \varphi : \mathcal S \slash \mathcal J \to \mathcal T$ given by $\widetilde \varphi (x+\mathcal J)=\varphi(x)$ is a unital completely positive map from Proposition 3.6 in the same paper. In particular, when $$M_n(\mathcal S \slash \mathcal J)^+ = \{ [x_{i,j}+\mathcal J]_{i,j} : \exists k_{i,j} \in \mathcal J,\ [x_{i,j}+k_{i,j}]_{i,j} \in M_n(\mathcal S)^+ \}$$ for all $n \in \mathbb N$, we call the kernel $\mathcal J$ completely order proximinal.

Since the kernel $\mathcal J$ in an operator system $\mathcal S$ is a closed subspace, the operator space structure of $\mathcal S \slash \mathcal J$ can be interpreted in two ways: first, as the operator space quotient and second, as the operator space structure induced by the operator system quotient. The two matrix norms can be different. For a specific example, see \cite[Example 4.4]{KPTT2}.

For a unital completely positive surjection $\varphi : \mathcal S \to \mathcal T$, we call $\varphi : \mathcal S \to \mathcal T$ {\sl a complete order quotient map} \cite[Definition 3.1]{H}  if for any $Q$ in $M_n(\mathcal T)^+$ and $\varepsilon>0$, we can take an element $P$ in $M_n(\mathcal S)$ so that it satisfies $$P + \varepsilon I_n \otimes 1_{\mathcal S} \in M_n(\mathcal S)^+ \quad \text{and} \quad \varphi_n(P) = Q,$$ or equivalently, if for any $Q$ in $M_n(\mathcal T)^+$ and $\varepsilon>0$, we can take a positive element $P$ in $M_n(\mathcal S)$ satisfying $$\varphi_n(P) = Q+\varepsilon I_n \otimes 1_{\mathcal S}.$$ This definition is compatible with \cite[Proposition 3.2]{FKP}: every strictly positive element lifts to a strictly positive element. An element $x \in \mathcal S$ is called {\em strictly positive} if there exists $\delta >0$ such that $x \ge \delta 1_{\mathcal S}$. The map $\varphi : \mathcal S \to \mathcal T$ is a complete order quotient map if and only if the induced map $\widetilde{\varphi} : \mathcal S \slash \ker \varphi \to \mathcal T$ is a unital complete order isomorphism. In other operator system references, this is termed a complete quotient map. To avoid confusion with complete quotient maps in operator space theory, we use the terminology of {\sl a complete order quotient map} throughout this paper. In this paper, we say that a linear map $\Phi : V \to W$ for operator spaces $V$ and $W$ is a {\sl complete quotient map} if $\Phi_n$ maps the open unit ball of $M_n(V)$ onto the open unit ball of $M_n(W)$. When $\varphi : \mathcal S \to \mathcal T$ is a complete order quotient map (respectively a complete order embedding), we will use the special type arrow as $\varphi : \mathcal S \twoheadrightarrow \mathcal T$ (respectively $\varphi : \mathcal S \hookrightarrow \mathcal T$) throughout the paper.

The normed space dual $\mathcal S^*$ of an operator system $\mathcal S$ is matrix ordered by the cones
$$
M_n(\mathcal S^*)^+ = \{ \text{completely positive maps from $\mathcal S$ to $M_n$} \},
$$
where we identify $[\varphi_{i,j}] \in M_n(\mathcal S^*)$ with the mapping $x \in \mathcal S \mapsto [\varphi_{i,j}(x)] \in M_n$.
Unfortunately, duals of operator systems fail to be operator systems in general
due to the lack of matrix order unit. When $\mathcal S$ is finite-dimensional, there exists a state $\omega_0$ on $\mathcal S$ such that $(\mathcal S^*, \{M_n(\mathcal S^*)^+\}_{n \in \mathbb N}, \omega_0)$ is an operator system \cite[Corollary 4.5]{CE}. In fact, we can show that every faithful state on $\mathcal S$ plays such a role by the compactness of $\mathcal S_{\|\cdot\|=1}^+$.

Let $f$ (respectively $g$) be a state on $M_n(\mathcal S)$ (respectively $M_n(\mathcal T)$).  We identify a subsystem $M_n \otimes \mathbb C 1_{\mathcal S}$ of $M_n(\mathcal S)$ (respectively $M_n \otimes \mathbb C 1_{\mathcal T}$ of $M_n(\mathcal T)$) with  $M_n$. We call $(f,g)$ a compatible pair whenever $f|_{M_n}=g|_{M_n}$. An operator system structure is defined on the amalgamated direct sum $\mathcal S \oplus \mathcal T \slash \langle (1_{\mathcal S},-1_{\mathcal T}) \rangle$ identifying each order unit. For $s \in M_n(\mathcal S)$ and $t \in M_n(\mathcal T)$, we define
\begin{enumerate}
\item $(s+t)^*=s^*+t^*$,
\item $s+t \ge 0$ if and only if $f(s)+g(t) \ge 0$ for all compatible pairs $(f,g)$.
\end{enumerate}
This operator system is denoted by $\mathcal S \oplus_1 \mathcal T$ and called the coproduct of operator systems $\mathcal S$ and $\mathcal T$. The canonical inclusion from $\mathcal S$ (respectively $\mathcal T$) into $\mathcal S \oplus_1 \mathcal T$ is a complete order embedding. The coproducts of operator systems satisfy the universal property: for unital completely positive maps $\varphi : \mathcal S \to \mathcal R$ and $\psi : \mathcal T \to \mathcal R$, there is a unique unital completely positive map $\Phi : \mathcal S \oplus_1 \mathcal T \to \mathcal R$ that extends both $\varphi$ and $\psi$, i.e. such that the diagram
$$\xymatrix{\mathcal S \ar@{_{(}->}[d] \ar[drr]^{\varphi} && \\ \mathcal S  \oplus_1 \mathcal T  \ar@{-->}[rr]^{\Phi}  & & \mathcal R  \\ \mathcal T \ar@{^{(}->}[u]   \ar[urr]_{\psi} & & }$$ commutes \cite[Proposition 3.3]{F}.
The coproduct $\mathcal S \oplus_1 \mathcal T$ can be realized as a quotient operator system. The map $$s + t  \in \mathcal S \oplus_1 \mathcal T \mapsto 2(s,t) +  \langle 1_{\mathcal S}, -1_{\mathcal T} \rangle \in \mathcal S \oplus \mathcal T \slash \langle 1_{\mathcal S}, -1_{\mathcal T} \rangle$$ is a unital complete order isomorphism \cite{Ka1}.

We refer to \cite{KPTT1}, \cite{KPTT2}, \cite{Ka1} and \cite{F} for general information on tensor products, quotients, duals and coproducts of operator systems.

\section{Universal operator systems $\mathfrak C_I$}

The coproduct of two operator systems can be generalized to any family of operator systems in a way parallel to \cite{F}. Suppose that $\{ \mathcal S_\iota\}_{\iota \in I}$ is a family of operator systems. We consider their algebraic direct sum $\oplus_{\iota \in I} \mathcal S_\iota$ consisting of finitely supported elements and its subspace $$N = {\rm span} \{ n_{\iota_1} - n_{\iota_2} \in \oplus_{\iota \in I} \mathcal S_\iota : \iota_1, \iota_2 \in I \},$$ where $$n_{\iota_0} (\iota) = \begin{cases} 1_{\mathcal S_{\iota_0}}, & \iota=\iota_0 \\ 0, & \text{otherwise}. \end{cases}$$ The algebraic quotient $$(\oplus_{\iota \in I} \mathcal S_\iota) \slash N$$ can be regarded as an amalgamated direct sum of $\{ \mathcal S_\iota \}_{\iota \in I}$ identifying all order units $1_{\mathcal S_\iota}$ over $\iota \in I$. We denote general elements in $M_n((\oplus_{\iota \in I} \mathcal S_\iota) \slash N)$ in brief by $$\sum_{\iota \in F} x_\iota, \qquad \text{where}~x_\iota \in M_n(\mathcal S_\iota),~ \text{$F$ is a finite subset of $I$}.$$

Let $\omega_\iota$ be a state on $M_n(\mathcal S_\iota)$ for each $\iota \in F$. We identify each subsystem $M_n \otimes \mathbb C 1_{\mathcal S_\iota}$ of $M_n(\mathcal S_\iota)$ with  $M_n$. Whenever $\omega_{\iota_1}|_{M_n} = \omega_{\iota_2}|_{M_n}$ for each $\iota_1, \iota_2 \in F$, we call the family $\{ \omega_\iota \}_{\iota \in F}$ compatible. On $M_n ((\oplus_{\iota \in I} \mathcal S_\iota) \slash N)$, we define the involution by $$(\sum_{\iota \in F} x_\iota)^* = \sum_{\iota \in F} x_\iota^*$$ and the positive cone as $$\sum_{\iota \in F} x_\iota \in M_n ((\oplus_{\iota \in I} \mathcal S_\iota) \slash N)^+ \quad \text{iff} \quad \sum_{\iota \in F} \omega_\iota (x_\iota) \ge 0$$ for any compatible family $\{ \omega_\iota \}_{\iota \in F}$ of states. The triple $$((\oplus_{\iota \in I} \mathcal S_\iota) \slash N, \{ M_n ((\oplus_{\iota \in I} \mathcal S_\iota) \slash N)^+ \}_{n \in \mathbb N}, n_\iota +N)$$ is denoted by $\oplus_1 \{ \mathcal S_\iota : \iota \in I \}$ and called the coproduct of operator systems $\{ \mathcal S_\iota \}_{\iota \in I}$. The following is an immediate generalization of the results on the coproduct of two operator systems studied in \cite{F} to any family of operator systems.

\begin{prop}\label{coproduct}
Suppose that  $\{ \mathcal S_\iota \}_{\iota \in I}$ (respectively $\{ \mathcal A_\iota \}_{\iota \in I}$) is a family of operator systems (respectively unital $C^*$-algebras). Then
\begin{enumerate}
\item[(i)] $\oplus_1 \{ \mathcal S_\iota : \iota \in I \}$ is an operator system;
\item[(ii)] for any subset $J \subset I$, the inclusion $$\oplus_1 \{ \mathcal S_\iota : \iota \in J \} \subset \oplus_1 \{ \mathcal S_\iota : \iota \in I \}$$ is completely order isomorphic;
\item[(iii)] for unital completely positive maps $\varphi_\iota : \mathcal S_\iota \to \mathcal R$, there exists unique unital completely positive map $\Phi : \oplus_1 \{ \mathcal S_\iota : \iota \in I \} \to \mathcal R$ which extends all $\varphi_\iota$, i.e. such that the diagram $$\xymatrix{\mathcal S_{\iota} \ar@{_{(}->}[d] \ar[drr]^{\varphi_{\iota}} && \\ \oplus_1 \{ \mathcal S_\iota : \iota \in I \} \ar@{-->}[rr]^{\Phi}  & & \mathcal R}$$ commutes;
\item[(iv)] we have $\sum_{\iota \in F} x_\iota \in M_n(\oplus_1 \{ \mathcal S_\iota : \iota \in I \})^+$ if and only if there exist $\alpha_{\iota} \in M_n$ for $\iota \in F$ such that $$\sum_{\iota \in F} \alpha_\iota = 0 \quad \text{and} \quad x_\iota + \alpha_\iota \otimes 1_{\mathcal S_\iota} \in M_n(\mathcal S_\iota)^+;$$
\item[(v)] the coproduct $\oplus_1 \{ \mathcal A_\iota : \iota \in I \}$ is an operator subsystem of the unital $C^*$-algebra free product $\ast_{\iota \in I} \mathcal A_\iota$.
\end{enumerate}
\end{prop}

The following is an immediate generalization of \cite[Proposition 4.7]{Ka2} to any finite family of operator systems.

 \begin{prop}\label{coproduct-quotient}
Suppose that $\mathcal S_1, \cdots, \mathcal S_n$ are operator systems and $$N= {\rm span} \{ n_i - n_j \in \mathcal S_1 \oplus \cdots \oplus \mathcal S_n : 1 \le i,j \le n \} \qquad (n_i(j) = \delta_{i,j} 1_{\mathcal S_i}).$$ Then, $N$ is a kernel in $\mathcal S_1 \oplus \cdots \oplus \mathcal S_n$ and the map $$\sum_{i=1}^n x_i \in \mathcal S_1 \oplus_1 \cdots \oplus_1 \mathcal S_n \mapsto n(x_i) + N \in \mathcal S_1 \oplus \cdots \oplus \mathcal S_n \slash N$$ is a unital complete order isomorphism.
\end{prop}

Suppose that $I$ is an index set and $\{ I_k \}_{k \in \mathbb N}$ is a sequence of index sets having the same cardinality as $I$. Let $M_k(C([0,1]))_{\iota_k}$ denote the copy of $M_k(C([0,1]))$ for each index $\iota_k \in I_k$. We denote the copy of $1 \in C([0,1])$ (respectively $t \in C([0,1])$) in $M_k(C([0,1])_{\iota_k}$ by $1_{\iota_k}$ (respectively $t_{\iota_k}$). For each $\iota_k \in I_k$, we let $\mathfrak C_{\iota_k}$ be an operator subsystem of $M_k(C([0,1])_{\iota_k}$ generated by $$\{ e_{ij} \otimes 1_{\iota_k} : 1 \le i,j \le k \} \quad \text{and} \quad \{ e_{ij} \otimes t_{\iota_k} :  1 \le i,j \le k \}.$$ We define the operator system $\mathfrak C_I$ as the coproduct $$\oplus_1 \{ \mathfrak C_{\iota_k} : k \in \mathbb N, \iota_k \in I_k \}.$$ The operator system $\mathfrak C_I$ depends only on the cardinality of the index set $I$.

\begin{prop}\label{matrix-coproduct}
The operator system $\mathfrak C_I$ is unitally completely order isomorphic to the coproduct  $$\oplus_1 \{ (M_k \oplus M_k)_{\iota_k} : k \in \mathbb N, \iota_k \in I_k \}.$$
\end{prop}

\begin{proof}
It is sufficient to show that each $\mathfrak C_{\iota_k}$ is unitally completely order isomorphic to the direct sum $M_k \oplus M_k$. For $\alpha, \beta \in M_{nk}$, we have
\begin{enumerate}
\item[] $\alpha \otimes 1_{\iota_k} +\beta \otimes t_{\iota_k}$ is positive in $M_n(\mathfrak C_{\iota_k})$
\item[$\Leftrightarrow$] $\alpha \otimes 1 +\beta \otimes t$ is positive in $M_n(M_k(C([0,1])))$
\item[$\Leftrightarrow$] $\forall t \in [0,1], f(t)=\alpha +t\beta \in M_{nk}^+$ (since $M_n(M_k(C([0,1]))) \simeq C([0,1],M_{nk})$)
\item[$\Leftrightarrow$] $\alpha, \alpha + \beta \in M_{nk}^+$ (because $f$ is affine).
\end{enumerate}
Hence, the mapping $$\alpha \otimes 1_{\iota_k} + \beta \otimes t_{\iota_k} \in \mathfrak C_{\iota_k} \mapsto (\alpha , \alpha + \beta) \in M_k \oplus M_k, \qquad \alpha, \beta \in M_k$$ is a unital complete order isomorphism.
\end{proof}

A $C^*$-cover $(\mathcal A, \iota)$ of an operator system $\mathcal S$ is a unital $C^*$-algebra $\mathcal A$ with a unital complete order embedding $\iota : \mathcal S \hookrightarrow \mathcal A$ such that $\iota (\mathcal S)$ generates $\mathcal A$ as a $C^*$-algebra. The enveloping $C^*$-algebra $C^*_e(\mathcal S)$ is a $C^*$-cover of $\mathcal S$ satisfying the universal minimal property: for any $C^*$-cover $\iota : \mathcal S \hookrightarrow \mathcal A$, there is a unique unital $*$-homomorphism $$\pi : \mathcal A \to C^*_e(\mathcal S)$$ such that $\pi(\iota (x))=x$ for all $x \in \mathcal S$ \cite{Ham}.

Let $\mathcal S$ be an operator subsystem of $\mathcal T$. We say that $\mathcal S$ is relatively weakly injective in $\mathcal T$ if $$\mathcal S \otimes_{\rm c} \mathcal R \hookrightarrow \mathcal T \otimes_{\rm c} \mathcal R$$
for any operator system $\mathcal R$. The following are equivalent \cite[Theorem 4.1]{Bh}:
\begin{enumerate}
\item[(i)] $\mathcal S$ is relatively weakly injective in $\mathcal T$;
\item[(ii)] $\mathcal S \otimes_{\rm c} C^*(\mathbb F_\infty) \hookrightarrow \mathcal T \otimes_{\rm c} C^*(\mathbb F_\infty)$;
\item[(iii)] for any unital completely positive map $\varphi : \mathcal S \to B(H)$, there exists a unital completely positive map $\Phi : \mathcal T \to \varphi(\mathcal S)''$ such that $\Phi|_{\mathcal S} = \varphi$.
\end{enumerate}

\begin{thm}\label{cover}
Suppose that $I$ is an index set and $\{ I_k \}_{k \in \mathbb N}$ is a sequence of index sets having the same cardinality as $I$. Then,
\begin{enumerate}
\item[(i)] the unital $C^*$-algebra free product  $$\ast_{k \in \mathbb N, \iota_k \in I_k} M_k(C([0,1]))_{\iota_k}$$ is a $C^*$-cover of $\mathfrak C_I$;
\item[(ii)] the unital $C^*$-algebra free product $$\ast_{k \in \mathbb N, \iota_k \in I_k} (M_k \oplus M_k)_{\iota_k}$$ is a $C^*$-envelope of $\mathfrak C_I$;
\item[(iii)] for a unital $C^*$-algebra $\mathcal A$, every unital completely positive map $\varphi : \mathfrak C_I \to \mathcal A$ has completely positive extensions $$\Phi : \ast_{k \in \mathbb N, \iota_k \in I_k} M_k(C([0,1]))_{\iota_k} \to \mathcal A \quad \text{and} \quad \Psi : \ast_{k \in \mathbb N, \iota_k \in I_k} (M_k \oplus M_k)_{\iota_k} \to \mathcal A;$$
\item[(iv)] $\mathfrak C_I$ is relatively weakly injective in both $$\ast_{k \in \mathbb N, \iota_k \in I_k} M_k(C([0,1]))_{\iota_k} \quad \text{and} \quad \ast_{k \in \mathbb N, \iota_k \in I_k} (M_k \oplus M_k)_{\iota_k}.$$
\end{enumerate}
\end{thm}

\begin{proof}
(i) By Proposition \ref{coproduct} (iv), (v), we have
\begin{enumerate}
\item[] $\sum_{\iota_k \in F} x_{\iota_k} \in M_n(\mathfrak C_I)^+$
\item[$\Leftrightarrow$] $\exists \alpha_{\iota_k} \in M_n$, $\sum_{\iota_k \in F} \alpha_{\iota_k} = 0$ and $x_{\iota_k} +\alpha_{\iota_k} \otimes 1_{\mathfrak C_{\iota_k}} \in M_n(\mathfrak C_{\iota_k})^+$
\item[$\Leftrightarrow$] $\exists \alpha_{\iota_k} \in M_n$, $\sum_{\iota_k \in F} \alpha_{\iota_k} = 0$ and $x_{\iota_k} +\alpha_{\iota_k} \otimes 1_{\mathfrak C_{\iota_k}} \in M_n(M_k(C([0,1]))_{\iota_k})^+$
\item[$\Leftrightarrow$] $\sum_{\iota_k \in F} x_{\iota_k} \in M_n(\oplus_1 \{ M_k(C([0,1]))_{\iota_k} : k \in \mathbb N, \iota_k \in I_k \})^+$
\item[$\Leftrightarrow$]  $\sum_{\iota_k \in F} x_{\iota_k} \in M_n(\ast_{k \in \mathbb N, \iota_k \in I_k} M_k(C([0,1]))_{\iota_k})^+$.
\end{enumerate}
Hence, $\mathfrak C_I$ is an operator subsystem of $\ast_{k \in \mathbb N, \iota_k \in I_k} M_k(C([0,1]))_{\iota_k}$. By the Weierstrass approximation theorem, each $\mathfrak C_{\iota_k}$ generates $ M_k(C([0,1]))_{\iota_k}$ as a $C^*$-algebra. Hence, $\ast_{k \in \mathbb N, \iota_k \in I_k} M_k(C([0,1]))_{\iota_k}$ is a $C^*$-cover of $\mathfrak C_I$.

(ii) The proof is motivated by \cite[Theorem 2.6]{FP}. Suppose that $$\mathfrak C_I \subset B(H) \qquad \text{and} \qquad \ast_{k \in \mathbb N, \iota_k \in I_k} (M_k \oplus M_k)_{\iota_k} \subset B(K).$$ Let $\mathcal A$ be a $C^*$-algebra generated by $\mathfrak C_I$ in $B(H)$. By the Arveson extension theorem, the canonical inclusion from $\mathfrak C_I =\oplus_1 \{ (M_k \oplus M_k)_{\iota_k} : k \in \mathbb N, \iota_k \in I_k \}$ into $\ast_{k \in \mathbb N, \iota_k \in I_k} (M_k \oplus M_k)_{\iota_k}$ extends to a unital completely positive map $\rho : \mathcal A \to B(K)$. Then letting $\rho =V^* \pi (\cdot) V$ be a minimal Stinespring decomposition of $\rho$ for a $*$-representation $\pi : \mathcal A \to B(\widehat{K})$ and an isometry $V : K \to \widehat{K}$, we have the commutative diagram $$\xymatrix{\mathcal A \ar[rr]^{\pi} \ar[drr]^{\rho} && B(\widehat{K}) \ar[d]^{V^* \cdot\ V} \\ \mathfrak C_I \ar@{^{(}->}[u] \ar@{}|-*[@]{\subset}[r] & \ast_{k \in \mathbb N, \iota_k \in I_k} (M_k \oplus M_k)_{\iota_k} \ar@{}|-*[@]{\subset}[r] & B(K).}$$

For a unitary matrix $U$ in $(M_k \oplus M_k)_{\iota_k}$ ($U$ need not be unitary in $\mathcal A$), we can write $\pi (U)$ in the operator matrix form $$\pi (U)=\begin{pmatrix} U & B \\ C & D \end{pmatrix}.$$ Since $U$ is unitary in $B(K)$ and $$1 = \|U\| \le \| \begin{pmatrix} U & B \\ C & D \end{pmatrix} \| = \|\pi(U) \| \le 1,$$ we have $B=0=C$ by the $C^*$-axiom. It follows that $\rho$ is multiplicative on $$\mathcal U := \{ U \in (U(k) \oplus U(k))_{\iota_k} : k \in \mathbb N, \iota_k \in I_k \}.$$ By the spectral theorem, every matrix can be written as a linear combination of unitary matrices. It follows that the set $\mathcal U$ generates $\mathcal A$ as a $C^*$-algebra. We can regard $\rho$ as a surjective $*$-homomorphism from $\mathcal A$ onto $\ast_{k \in \mathbb N, \iota_k \in I_k} (M_k \oplus M_k)_{\iota_k}$. Hence, $\ast_{k \in \mathbb N, \iota_k \in I_k} (M_k \oplus M_k)_{\iota_k}$ is the universal quotient of all $C^*$-algebras generated by $\mathfrak C_I$.

(iii) Since each $\mathfrak C_{\iota_k}$ is unitally completely order isomorphic to $M_k \oplus M_k$ which is injective, there exists a unital completely positive projection $$P_{\iota_k} : M_k(C([0,1]))_{\iota_k} \to \mathfrak C_{\iota_k}.$$ By \cite[Theorem 3.1]{B}, the unital free products $$\ast_{k \in \mathbb N, \iota_k \in I_k} (\varphi|_{\mathfrak C_{\iota_k}} \circ P_{\iota_k}) : \ast_{k \in \mathbb N, \iota_k \in I_k} M_k(C([0,1]))_{\iota_k} \to \mathcal A$$
and
$$\ast_{k \in \mathbb N, \iota_k \in I_k} \varphi|_{\mathfrak C_{\iota_k}} : \ast_{k \in \mathbb N, \iota_k \in I_k} (M_k \oplus M_k)_{\iota_k} \to \mathcal A$$ are completely positive extensions of $\varphi$.

(iv) Let $\varphi : \mathfrak C_I \to B(H)$ be a unital completely positive map. The double commutant $\varphi (\mathcal S)''$ of its range is a $C^*$-algebra. The relative weak injectivity follows from (iii) and \cite[Theorem 4.1]{Bh}.
\end{proof}

\begin{thm}\label{universal}
Suppose that $\mathcal S$ is an operator system and $\mathcal S^+_{\|\cdot\| \le 1}$ is indexed by a set $I$. Then, $\mathcal S$ is an operator system quotient of $\mathfrak C_I$. Furthermore, the kernel is completely order proximinal and every positive element $x \in M_k(\mathcal S)$ can be lifted to a positive element $\tilde{x} \in M_k(\mathfrak C_I)$ with $\|\tilde{x}\| \le k^2 \|x\|$.
\end{thm}

\begin{proof}
Let $\{ I_k \}_{k \in \mathbb N}$ be a sequence of index sets with the same cardinality as $I$. Then each element in $M_k(\mathcal S)_{\|\cdot \| \le 1}^+$ can be indexed by $I_k$. Suppose that $\mathcal S \subset B(H)$. Then for each index $\iota_k \in I_k$, we define a unital completely positive map $\Phi_{\iota_k} : M_k (C([0,1]))_{\iota_k} \to B(H)$ as $$\Phi_{\iota_k} (\alpha \otimes f) = {1 \over k} \begin{pmatrix} e_1^t & \cdots & e_k^t \end{pmatrix} \alpha \otimes f(x_{\iota_k}) \begin{pmatrix} e_1 \\ \vdots \\ e_k \end{pmatrix} = {1 \over k} \sum_{i,j} \alpha_{i,j} f(x_{\iota_k})_{i,j},$$ where $x_{\iota_k} \in M_k(\mathcal S)_{\|\cdot \| \le 1}^+$ and each $e_i$ is a column vector. Let $\varphi_{\iota_k} : \mathfrak C_{\iota_k} \to \mathcal S$ be its restriction on $\mathfrak C_{\iota_k}$. By Proposition \ref{coproduct} (iii), there exists a unital completely positive map  $\Phi : \mathfrak C_I \to \mathcal S$ which extends all $\varphi_{\iota_k}$ over $\iota_k \in I_k, k \in \mathbb N$. Since $\mathcal S_{\|\cdot\| \le 1}^+$ is contained in the range of $\Phi$, $\Phi$ is surjective.

Choose an element $x_{\iota_k} \in M_k(\mathcal S)^+_{\|\cdot\|=1}$. From $$\Phi_k (k[E_{ij} \otimes t_{\iota_k}]_{i,j})=[k\Phi (E_{ij} \otimes t_{\iota_k})]_{i,j}=[x_{\iota_k}(i,j)]_{i,j}=x_{\iota_k}$$ and $$k[E_{i,j} \otimes t_{\iota_k}]_{i,j}=k[E_{ij}]_{i,j} \otimes t_{\iota_k} = k \begin{pmatrix} e_1 \\ \vdots \\ e_n \end{pmatrix} \begin{pmatrix} e_1^t & \cdots & e_n^t \end{pmatrix} \otimes t_{\iota_k} \in M_{k^2}(C([0,1]))^+,$$ we see that $\Phi : \mathfrak C_I \to \mathcal S$ is a complete order quotient map whose kernel is completely order proximinal. Moreover, we have $$\begin{aligned} \|k[E_{i,j} \otimes t_{\iota_k}]_{i,j}\|_{M_k(C([0,1]))} & = \|k[E_{i,j}]_{i,j}\| \\ & =k\| \begin{pmatrix} e_1 \\ \vdots \\ e_k \end{pmatrix} \begin{pmatrix} e_1^t & \cdots & e_k^t \end{pmatrix}\| \\ & = k\| \begin{pmatrix} e_1^t & \cdots & e_k^t \end{pmatrix} \begin{pmatrix} e_1 \\ \vdots \\ e_k \end{pmatrix} \| \\ & = k^2. \end{aligned}$$
\end{proof}

We define the operator system $\mathfrak C_1$ as the coproduct $$\oplus_1 \{ M_k \oplus M_k : k \in \mathbb N \}.$$
Note that $\mathfrak C_1 = \mathfrak C_I$ when $|I|=1$.

\begin{thm}\label{N}
Suppose that an operator system $\mathcal S$ is a countable union of its finite dimensional subsystems. Then, $\mathcal S$ is an operator system quotient of $\mathfrak C_1$.
\end{thm}

\begin{proof}
First, we show that every finite dimensional operator system is an operator system quotient of $\mathfrak C_{\mathbb N}$. Let $E$ be a finite dimensional operator system. We index a countable dense subset $D_k$ of $M_k(E)_{\|\cdot\| \le 1}^+$ by $\mathbb N$. Define a unital completely positive map $\Phi : \mathfrak C_{\mathbb N} \to E$ as in Theorem \ref{universal}. Since the range of $\Phi$ is a dense subspace of a finite dimensional space $E$, $\Phi$ is surjective.

Choose $\varepsilon >0$ and an element $x$ in $M_k(E)_{\|\cdot\| \le 1}^+$. Since $E$ is finite dimensional, the inverse of $\tilde{\Phi} : \mathfrak C_{\mathbb N} \slash {\rm Ker} \Phi \to E$ is completely bounded. Let $$\|\tilde{\Phi}^{-1} : E \to \mathfrak C_{\mathbb N} \slash {\rm Ker} \Phi \|_{cb} \le M.$$ Take $y \in D_k$ so that $\|x-y\| \le {\varepsilon \over 2M}$. Since $\|\tilde{\Phi}^{-1}_k (x-y)\| \le {\varepsilon \over 2}$, we have $$\tilde{\Phi}_k^{-1} (x-y) + {\varepsilon \over 2} I_k \otimes 1_{\mathfrak C_{\mathbb N} \slash {\rm Ker} \Phi} \in M_k(\mathfrak C_{\mathbb N} \slash {\rm Ker} \Phi)^+.$$ There exists a positive element $z$ in $M_k(\mathfrak C_{\mathbb N})$ satisfying $$z + {\rm Ker} \Phi_k = \tilde{\Phi}^{-1}_k (x-y)+ \varepsilon I_k \otimes 1_{\mathfrak C_{\mathbb N} \slash {\rm Ker} \Phi},$$ which implies $$\Phi_k (z)= \tilde{\Phi}_k (z + {\rm Ker} \Phi) = x-y+\varepsilon I_k \otimes 1_E.$$ As in the proof of Theorem \ref{universal}, we can take a positive element $\tilde{y}$ in $M_k(\mathfrak C_{\mathbb N})$ such that $\Phi_k (\tilde y) = y$. It follows that $$\Phi_k (z + \tilde{y}) = (x-y+\varepsilon I_k \otimes 1_E )+y = x+\varepsilon I_k \otimes 1_E.$$ Hence, $E$ is an operator system quotient of $\mathfrak C_{\mathbb N}$.

Next, we show that $\mathfrak C_{\mathbb N}$ is an operator system quotient of $\mathfrak C_1$. We enumerate the coproduct summands of $\mathfrak C_{\mathbb N}$ as
$$(M_1 \oplus M_1)_1, (M_1 \oplus M_1)_2, (M_2 \oplus M_2)_1, (M_1 \oplus M_1)_3, (M_2 \oplus M_2)_2, (M_3 \oplus M_3)_1, \cdots$$ and denote them by $M_{a_k} \oplus M_{a_k}$. Since $k \ge a_k$, the identity map on $M_{a_k}$ is factorized as $Q_k \circ J_k$ for unital completely positive maps
$$J_k : A \in M_{a_k} \mapsto A \oplus \omega(A) I_{k-a_k} \in M_k \qquad (\omega : \text{a state on}~M_{a_k})$$ and
$$Q_k : A \in M_k \mapsto [A_{i,j}]_{1 \le i,j \le a_k} \in M_{a_k}.$$
By the universal property of the coproduct, there exists a unital completely positive map $J : \mathfrak C_{\mathbb N} \to \mathfrak C_1$ (respectively $Q : \mathfrak C_1 \to \mathfrak C_{\mathbb N}$) which extends all $J_k \oplus J_k :  M_{a_k} \oplus M_{a_k} \to M_k \oplus M_k$ (respectively $Q_k \oplus Q_k : M_k \oplus M_k \to M_{a_k} \oplus M_{a_k}$). Then, the identity map on $\mathfrak C_{\mathbb N}$ is factorized as $Q \circ J$. Hence, $\mathfrak C_{\mathbb N}$ is an operator system quotient of $\mathfrak C_1$.

Suppose that $\mathcal S = \bigcup_{k=1}^\infty E_k$ for finite dimensional subsystems $E_k$ of $\mathcal S$. We can find complete order quotient maps $\Psi_k : \mathfrak C_1 \to E_k$. By the universal property of the coproduct, there exists a unital completely positive map $\Psi : \mathfrak C_{\mathbb N} \to \mathcal S$ which extends all $\Psi_k$. It is easy to check that $\Psi$ is a complete order quotient map. Since $\mathfrak C_{\mathbb N}$ is an operator system quotient of $\mathfrak C_1$, $\mathcal S$ is an operator system quotient of $\mathfrak C_1$.
\end{proof}

\begin{thm}\label{LP}
Suppose that $\mathcal A$ is a unital $C^*$-algebra and $\mathcal I$ is a closed ideal in it. Every unital completely positive map $\varphi : \mathfrak C_I \to \mathcal A \slash \mathcal I$ lifts to a unital completely positive map $\tilde{\varphi} : \mathfrak C_I \to \mathcal A$, i.e. such that the diagram
$$\xymatrix{ & \mathcal A \ar@{->>}[d] \\ \mathfrak C_I \ar@{-->}[ur]^{\tilde{\varphi}} \ar[r]^{\varphi}& \mathcal A \slash \mathcal I}$$ commutes.
\end{thm}

\begin{proof}
Let $z_{\iota_k}$ be the direct sum of two Choi matrices associated to the restrictions of $\varphi|_{(M_k \oplus M_k)_{\iota_k}}$ on each two blocks $M_k$, that is, $$z_{\iota_k} = [\varphi|_{(M_k \oplus M_k)_{\iota_k}} (E_{i,j} \oplus 0_k)]_{i,j} \oplus [\varphi|_{(M_k \oplus M_k)_{\iota_k}} (0_k  \oplus E_{i,j})]_{i,j}.$$ Then $z_{\iota_k}$ belongs to the positive cone of ${M_k(\mathcal A \slash I)} \oplus M_k(\mathcal A \slash I)$. Then let $\tilde{z}_{\iota_k} \in M_k(\mathcal A) \oplus M_k(\mathcal A)$ be a positive lifting $z_{\iota_k}$. Its corresponding mapping $$\tilde{\varphi}_{\iota_k} : (M_k \oplus M_k)_{\iota_k} \to \mathcal A$$ is a completely positive lifting of $\varphi|_{(M_k \oplus M_k)_{\iota_k}}$. We let $$\tilde{\varphi}_{\iota_k}(I_{2k})=1+h, \qquad h=h^+-h^- \qquad (h \in \mathcal I, \quad h^+, h^- \in \mathcal I^+)$$ and take a state $\omega$ on $(M_k \oplus M_k)_{\iota_k}$. Considering $$\alpha \in (M_k \oplus M_k)_{\iota_k} \mapsto (1 + h^+)^{-{1 \over 2}}(\tilde{\varphi}_{\iota_k} (\alpha) + \omega (\alpha)h^-)(1 + h^+)^{-{1 \over 2}} \in \mathcal A$$ as in \cite[Remark 8.3]{KPTT2}, we may assume that the lifting $\tilde{\varphi}_{\iota_k}$ is unital.  By the universal property of the coproduct, there exists a unital completely positive map $\tilde{\varphi} : \mathfrak C_I \to \mathcal S$ that extends all $\tilde{\varphi}_{\iota_k}$.
\end{proof}

The universal $C^*$-algebra $C^*_u(\mathcal S)$ is the $C^*$-cover of $\mathcal S$ satisfying the universal property:
if $\varphi : \mathcal S \to \mathcal A$ is a unital completely positive map for a unital $C^*$-algebra $\mathcal A$, then there exists a $*$-homomorphism $\pi : C^*_u(\mathcal S) \to \mathcal A$ such that $\pi \circ \iota = \varphi$ \cite{KW}. For a unital completely positive map $\varphi : \mathcal S \to \mathcal T$ and a complete order embedding $\iota : \mathcal T \to C^*_u(\mathcal T)$, we denote the unique $*$-homomorphic extension of $\iota \circ \varphi : \mathcal S \to C^*_u(\mathcal T)$ by $C^*_u(\varphi)$. We can regard $C^*_u(\cdot)$ as a functor from the category of operator systems to the category of $C^*$-algebras.

\begin{cor}
Let $\mathcal S$ be an operator system and $Q : \mathfrak C_I \to \mathcal S$ be a complete order quotient map. The following are equivalent:
\begin{enumerate}
\item[(i)] $\mathcal S$ has the operator system lifting property;
\item[(ii)] $C^*_u(Q) : C^*_u(\mathfrak C_I) \to C^*_u(\mathcal S)$ has a unital $*$-homomorphic right inverse;
\item[(iii)] $C^*_u(Q) : C^*_u(\mathfrak C_I) \to C^*_u(\mathcal S)$ has a unital completely positive right inverse.
\end{enumerate}
\end{cor}

\begin{proof}
(i) $\Rightarrow$ (ii). The inclusion $\iota : \mathcal S \subset C^*_u(\mathcal S)$ lifts to a unital completely positive map $\tilde{\iota} : \mathcal S \to C^*_u(\mathfrak C_I)$. Its $*$-homomorphic extension $\rho : C^*_u(\mathcal S) \to C^*_u(\mathfrak C_I)$ is the right inverse of $C^*_u(Q) : C^*_u(\mathfrak C_I) \to C^*_u(\mathcal S)$.

(ii) $\Rightarrow$ (iii). Trivial.

(iii) $\Rightarrow$ (i). Suppose that $\varphi : \mathcal S \to \mathcal A \slash \mathcal I$ is a unital completely positive map for a unital $C^*$-algebra $\mathcal A$ and its closed ideal $\mathcal I$. By Theorem \ref{LP}, $\varphi \circ Q : \mathfrak C_I \to \mathcal A \slash \mathcal I$ lifts to a unital completely positive map $\psi : \mathfrak C_I \to \mathcal A$. Let $\rho : C^*_u(\mathfrak C_I) \to \mathcal A$ (respectively $\sigma : C^*_u(\mathcal S) \to {\mathcal A \slash \mathcal I}$) be a unique $*$-homomorphic extension of $\psi$ (respectively $\varphi$). Suppose that $r$ is a unital completely positive right inverse of $C^*_u(Q)$. We thus have the diagram $$\xymatrix{C^*_u(\mathfrak C_I) \ar[rr]^{\rho} \ar@<0.5ex>[rd]^{C^*_u(Q)} && \mathcal A \ar@{->>}[dd]^{\pi} \\ & C^*_u(\mathcal S) \ar[rd]^{\sigma} \ar@<0.5ex>[lu]^{r} & \\ \mathfrak C_I \ar@{^{(}->}[uu] \ar@{->>}[r]^{Q} & \mathcal S \ar@{^{(}->}[u]^{\iota} \ar[r]^{\varphi} & \mathcal A \slash \mathcal I.}$$ Let us show that $$\tilde{\varphi}:=\rho \circ r \circ \iota : \mathcal S \to \mathcal A$$ is a lifting of $\varphi$. Since $\mathfrak C_I$ generates $C^*_u(\mathfrak C_I)$ as a $C^*$-algebra, $\pi \circ \psi = \varphi \circ Q$ implies that $$\pi \circ \rho = \sigma \circ C^*_u (Q).$$ For $x \in \mathcal S$, we have $$\pi \circ \tilde{\varphi} (x)= \pi \circ \rho \circ r (x) = \sigma \circ C^*_u(Q) \circ r (x) = \varphi (x).$$
\end{proof}

\section{A Kirchberg type tensor theorem for operator systems}

For a free group $\mathbb F$ and a Hilbert space $H$, Kirchberg \cite[Corollary 1.2]{Ki} proved that $$C^*(\mathbb F) \otimes_{\min} B(H) = C^*(\mathbb F) \otimes_{\max} B(H).$$ Kirchberg's theorem is striking if we recall that $C^*(\mathbb F)$ and $B(H)$ are universal objects in the $C^*$-algebra category: every $C^*$-algebra is a $C^*$-quotient of $C^*(\mathbb F)$ and a $C^*$-subalgebra  of $B(H)$ for suitable choices of $\mathbb F$ and $H$. Every operator system is a quotient of $\mathfrak C_I$ and a subsystem of $B(H)$ for suitable choices of $I$ and $H$. Hence we may say that $$\mathfrak C_I \otimes_{\min} B(H) = \mathfrak C_I \otimes_{\max} B(H),$$ the proof of which will follow, is the Kirchberg type theorem in the category of operator systems.

If $\mathcal S$ has the operator system local lifting property, $\mathcal S \otimes_{\min} B(H) = \mathcal S \otimes_{\max} B(H)$ \cite[Theorem 8.6]{KPTT2}. From this, Theorem \ref{LP} immediately yields that $\mathfrak C_I \otimes_{\min} B(H) = \mathfrak C_I \otimes_{\max} B(H)$. The proof of \cite[Theorem 8.6]{KPTT2} depends on Kirchberg's theorem. We give a direct proof of $\mathfrak C_I \otimes_{\min} B(H) = \mathfrak C_I \otimes_{\max} B(H)$ that is independent of Kirchberg's theorem. By combining this with \cite{Ka2}, we present a new operator system theoretic proof of Kirchberg's theorem in Corollary \ref{Kirchberg}.

\begin{thm}\label{main}
For an index set $I$ and a Hilbert space $H$, we have $$\mathfrak C_I \otimes_{\min} B(H) = \mathfrak C_I \otimes_{\max} B(H).$$
\end{thm}

\begin{proof}
Let $z$ be a positive element in $\mathfrak C_I \otimes_{\min} B(H)$. We write $z=\sum_{\iota_k \in F} z_{\iota_k}$ for a finite subset $F$ of $\bigcup_{k=1}^\infty I_k$ and $z_{\iota_k} \in \mathfrak C_{\iota_k} \otimes B(H)$. By Proposition \ref{coproduct} (ii) and the injectivity of the minimal tensor product, we can regard $z$ as a positive element in $$\oplus_1 \{ (M_k \oplus M_k)_{\iota_k} : \iota_k \in F \} \otimes_{\min} B(H).$$

We apply Proposition \ref{coproduct-quotient} to $\oplus_1 \{ (M_k \oplus M_k)_{\iota_k} : \iota_k \in F \}$ to obtain the complete order isomorphism $$\Phi : \sum_{\iota_k \in F} x_{\iota_k} \in \oplus_1 \{ (M_k \oplus M_k)_{\iota_k} : \iota_k \in F \} \mapsto |F| (x_{\iota_k})_{\iota_k \in F} +N \in \oplus_{\iota_k \in F} (M_k \oplus M_k)_{\iota_k} \slash N,$$ where $|F|$ denotes the number of elements of the set $F$ and
$$N= {\rm span} \{ n_{\iota_l} - n_{\iota_m'} \in \oplus_{\iota_k \in F} (M_k \oplus M_k)_{\iota_k} : \iota_l, \iota_m' \in F\} \qquad (n_{\iota_k}(\iota_j') = \delta_{\iota_k, \iota_j'} I_k \oplus I_k).$$ Let $$Q : \oplus_{\iota_k \in F} (M_k \oplus M_k)_{\iota_k} \twoheadrightarrow \oplus_{\iota_k \in F} (M_k \oplus M_k)_{\iota_k} \slash N$$ be the canonical quotient map. By \cite[Proposition 1.15]{FP}, its dual map $$Q^* : (\oplus_{\iota_k \in F} (M_k \oplus M_k)_{\iota_k} \slash N)^* \hookrightarrow (\oplus_{\iota_k \in F} (M_k \oplus M_k)_{\iota_k})^*$$ is a complete order embedding. The range of $Q^*$ is the annihilator $$N^\perp = \{ \varphi \in (\oplus_{\iota_k \in F} (M_k \oplus M_k)_{\iota_k})^* : N \subset {\rm Ker} \varphi \}.$$

The linear map $\gamma_k : M_k \to M_k^*$ defined as $$\gamma_k (\alpha) (\beta) = \sum_{i,j=1}^k \alpha_{i,j} \beta_{i,j} = {\rm tr} (\alpha \beta^t)$$ is a complete order isomorphism \cite[Theorem 6.2]{PTT}. Define a complete order isomorphism
$$\Gamma : \oplus_{\iota_k \in F} (M_k \oplus M_k)_{\iota_k} \to \oplus_{\iota_k \in F} (M_k^* \oplus M_k^*)_{\iota_k} \simeq (\oplus_{\iota_k \in F} (M_k \oplus M_k)_{\iota_k})^*$$ by
$$\langle \Gamma ((\alpha_{\iota_k})), (\beta_{\iota_k}) \rangle = \langle ((\gamma_k \oplus \gamma_k) ({\alpha_{\iota_k}\over 2k})),(\beta_{\iota_k}) \rangle = \sum_{\iota_k \in F} {1 \over 2k} {\rm tr} ( \alpha_{\iota_k} \beta_{\iota_k}^t).$$
Then, $\Gamma^{-1}$ maps the annihilator $N^\perp$ onto the operator subsystem $$K=\{ (\alpha_{\iota_k}) \in \oplus_{\iota_k \in F} (M_k \oplus M_k)_{\iota_k} : {{\rm tr}(\alpha_{\iota_l}) \over l}= {{\rm tr}(\alpha_{\iota_m'}) \over m} \ \text{for all}\ \iota_l, \iota_m' \in F\}$$ of $\oplus_{\iota_k \in F} (M_k \oplus M_k)_{\iota_k}$. We have obtained complete order isomorphisms
$$(\oplus_1 \{ (M_k \oplus M_k)_{\iota_k} : \iota_k \in F \})^* \simeq (\oplus_{\iota_k \in F} (M_k \oplus M_k)_{\iota_k} \slash N)^* \simeq N^\perp \simeq K.$$

Considering the duals of the above isomorphisms, we obtain a complete order isomorphism $$\Lambda : \oplus_1 \{ (M_k \oplus M_k)_{\iota_k} : \iota_k \in F \} \to K^*,$$ which maps each $\sum_{\iota_k \in F} \beta_{\iota_k} \in \oplus_1 \{ (M_k \oplus M_k)_{\iota_k} : \iota_k \in F \}$ to a functional $$(\alpha_{\iota_k}) \in K \mapsto \sum_{\iota_k \in F} {|F| \over 2k} {\rm tr} ( \beta_{\iota_k} \alpha_{\iota_k}^t) \in \mathbb C.$$
In particular, $\Lambda$ maps the order unit to the state $\omega$ on $K$ defined as $$\omega ((\alpha_{\iota_k}))=\sum_{\iota_k \in F} {1 \over 2k} {\rm tr} (\alpha_{\iota_k}).$$

It enables us to identify $$\oplus_1 \{ (M_k \oplus M_k)_{\iota_k} : \iota_k \in F \} \otimes_{\min} B(H) \simeq K^* \otimes_{\min} B(H),$$ where $K^*$ is an operator system with an order unit $\omega$. The linear map $\varphi : K \to B(H)$ corresponding to $z$ in a canonical way is completely positive \cite[Lemma 8.5]{KPTT2}. By the Arveson extension theorem, $\varphi : K \to B(H)$ extends to a completely positive map $\tilde{\varphi} : \oplus_{\iota_k \in F} (M_k \oplus M_k)_{\iota_k} \to B(H)$.
We have the commutative diagram
$$\xymatrix{\oplus_{\iota_k \in F} (M_k \oplus M_k)_{\iota_k} \ar[rr]^-{\Phi^{-1} \circ Q} \ar[d]_{\Gamma} && \oplus_1 \{ (M_k \oplus M_k)_{\iota_k} : \iota_k \in F \} \ar[d]^\Lambda \\ (\oplus_{\iota_k \in F} (M_k \oplus M_k)_{\iota_k})^* \ar[rr]^-R && K^*},$$ where $R$ denotes the restriction.
It follows that $$(\Phi^{-1} \circ Q) \otimes {\rm id} : \oplus_{\iota_k \in F} (M_k \oplus M_k)_{\iota_k} \otimes_{\min} B(H) \to \oplus_1 \{ (M_k \oplus M_k)_{\iota_k} : \iota_k \in F \} \otimes_{\min} B(H)$$ is a complete order quotient map. Maximal tensor products of complete order quotient maps are still complete order quotient maps \cite[Theorem 3.4]{H}. Hence, we obtain $$\xymatrix{\oplus_{\iota_k \in F} (M_k \oplus M_k)_{\iota_k} \otimes_{\min} B(H) \ar@{=}[r]\ar@{->>}[d]_{(\Phi^{-1} \circ Q) \otimes {\rm id}} & \oplus_{\iota_k \in F} (M_k \oplus M_k)_{\iota_k} \otimes_{\max} B(H) \ar@{->>}[d]^{(\Phi^{-1} \circ Q) \otimes {\rm id}} \\ \oplus_1 \{ (M_k \oplus M_k)_{\iota_k} : \iota_k \in F \} \otimes_{\min} B(H)  & \oplus_1 \{ (M_k \oplus M_k)_{\iota_k} : \iota_k \in F \} \otimes_{\max} B(H).}$$ The element $z$ is also positive in $\mathfrak C_I \otimes_{\max} B(H)$. The same arguments apply to all matricial levels.
\end{proof}

The maximal tensor product and the commuting tensor product are two different means of extending the $C^*$-maximal tensor product from the category of $C^*$-algebras to operator systems. For this reason, the weak expectation property of $C^*$-algebras bifurcates into the weak expectation property and the double commutant expectation property of operator systems. We say that an operator system $\mathcal S$ has the {\sl double commutant expectation property} provided that for every completely order isomorphic inclusion $\mathcal S \subset B(H)$, there exists a completely positive map $\varphi : B(H) \to \mathcal S''$ that fixes $\mathcal S$. For an operator system $\mathcal S$, the following are equivalent \cite[Theorem 7.6]{KPTT2}, \cite[Theorem 5.9]{Ka2}:
\begin{enumerate}
\item[(i)] $\mathcal S$ has the double commutant expectation property;
\item[(ii)] $\mathcal S$ is $({\rm el}, {\rm c})$-nuclear;
\item[(iii)] $\mathcal S \otimes_{\min} C^*(\mathbb F_\infty) = \mathcal S \otimes_{\max} C^*(\mathbb F_\infty)$;
\item[(iv)] $\mathcal S \otimes_{\min} (\ell^2_\infty \oplus_1 \ell^3_\infty) = \mathcal S \otimes_{\rm c} (\ell^2_\infty \oplus_1 \ell^3_\infty)$.
\end{enumerate}

\begin{thm}\label{DCEP}
An operator system $\mathcal S$ has the double commutant expectation property if and only if it satisfies $$\mathcal S \otimes_{\min} \mathfrak C_I = \mathcal S \otimes_{\rm c} \mathfrak C_I.$$
\end{thm}

\begin{proof}
$\Rightarrow )$ Every operator system with the double commutant expectation property is $({\rm el}, {\rm c})$-nuclear. Since the minimal tensor product is injective \cite[Theorem 4.6]{KPTT1}, we have $$\xymatrix{B(H) \otimes_{\min} \mathfrak C_I \ar@{=}[r] & B(H) \otimes_{\max} \mathfrak C_I \\ \mathcal S \otimes_{\min} \mathfrak C_I \ar@{^{(}->}[u] & \mathcal S \otimes_{{\rm el}={\rm c}} \mathfrak C_I. \ar@{^{(}->}[u]}$$

$\Leftarrow )$ Fix two indices $\iota_2' \in I_2$ and $\iota_3' \in I_3$. Define a unital completely positive map $\Phi : \ell^2_\infty \oplus_1 \ell^3_\infty \to \mathfrak C_I$ by $$\begin{aligned} & \Phi((a_1,a_2)+(b_1,b_2,b_3)) \\ = & {\rm diag}(a_1,a_2,a_1,a_2)+{\rm diag}(b_1,b_2,b_3,b_1,b_2,b_3) \in (M_2 \oplus M_2)_{\iota_2'} \oplus_1 (M_3 \oplus M_3)_{\iota_3'} \subset \mathfrak C_I.\end{aligned}$$ For each index $\iota_k \ne \iota_2', \iota_3'$, we take a state $\omega_{\iota_k}$ on $(M_k \oplus M_k)_{\iota_k}$ and define a unital completely positive map $$\psi_{\iota_k} : (M_k \oplus M_k)_{\iota_k} \to \ell^2_\infty \oplus_1 \ell^3_\infty$$ as $\psi_{\iota_k}(\alpha) = \omega_{\iota_k} (\alpha) 1_{\ell^2_\infty \oplus_1 \ell^3_\infty}$. For $\iota_2'$ and $\iota_3'$, we also define unital completely positive maps $$\psi_{\iota_2'} : (M_2 \oplus M_2)_{\iota_2'} \to \ell^2_\infty \oplus_1 \ell^3_\infty \qquad \text{and}\qquad \psi_{\iota_3'} : (M_3 \oplus M_3)_{\iota_3'} \to \ell^2_\infty \oplus_1 \ell^3_\infty$$ as  $\psi_{\iota_2'}(\alpha \oplus \beta)=(\alpha_{11}, \alpha_{22})$ and $\psi_{\iota_3'}(\alpha \oplus \beta)=(\alpha_{11}, \alpha_{22},\alpha_{33})$. By the universal property of the coproduct, there exists a unital completely positive map $\Psi : \mathfrak C_I \to \ell^2_\infty \oplus_1 \ell^3_\infty$ that extends all $\psi_{\iota_k}$. The identity map on $\ell^2_\infty \oplus_1 \ell^3_\infty$ is factorized through unital completely positive maps as $${\rm id}_{\ell^2_\infty \oplus_1 \ell^3_\infty} = \Psi \circ \Phi.$$

By the hypothesis, we have completely positive maps
$$\xymatrix{\mathcal S \otimes_{\min} (\ell^2_\infty \oplus_1 \ell^3_\infty) \ar[r]^-{{\rm id}_{\mathcal S} \otimes \Phi} & \mathcal S \otimes_{\min} \mathfrak C_I = \mathcal S \otimes_c \mathfrak C_I \ar[r]^-{{\rm id}_{\mathcal S} \otimes \Psi} & \mathcal S \otimes_{\rm c} (\ell^2_\infty \oplus_1 \ell^3_\infty).}$$
Since the positive cone of the commuting tensor product is the subcone of that of the minimal tensor product at each matrix level, we have $$\mathcal S \otimes_{\min} (\ell^2_\infty \oplus_1 \ell^3_\infty) = \mathcal S \otimes_{\rm c} (\ell^2_\infty \oplus_1 \ell^3_\infty).$$ By \cite[Theorem 5.9]{Ka2}, $\mathcal S$ has the double commutant expectation property.
\end{proof}

Since the maximal tensor product and the commuting tensor product are two different means of extending the $C^*$-maximal tensor product from the category of $C^*$-algebras to operator systems, we can regard $$\mathfrak C_I \otimes_{\min} \mathfrak C_I =\mathfrak C_I \otimes_{\max} \mathfrak C_I \qquad \text{and} \qquad \mathfrak C_I \otimes_{\min} \mathfrak C_I =\mathfrak C_I \otimes_c \mathfrak C_I$$ as operator system analogues of Kirchberg's conjecture $$C^*(\mathbb F) \otimes_{\min} C^*(\mathbb F) = C^*(\mathbb F) \otimes_{\max} C^*(\mathbb F).$$ The former is not true and the latter is equivalent to the Kirchberg's conjecture itself.

\begin{cor}\label{CC}
\begin{enumerate}
\item[(i)] $\mathfrak C_I \otimes_{c} \mathfrak C_I \ne \mathfrak C_I \otimes_{\max} \mathfrak C_I$. In particular, $\mathfrak C_I \otimes_{\min} \mathfrak C_I \ne \mathfrak C_I \otimes_{\max} \mathfrak C_I$.
\item[(ii)] The Kirchberg's conjecture has an affirmative answer if and only if $$\mathfrak C_I \otimes_{\min} \mathfrak C_I =\mathfrak C_I \otimes_c \mathfrak C_I.$$
\end{enumerate}
\end{cor}

\begin{proof}
(i) Similarly as the proof of Theorem \ref{DCEP}, we can show that the identity map on $\ell^2_\infty \oplus_1 \ell^2_\infty$ is factorized as $${\rm id}_{\ell^2_\infty \oplus_1 \ell^2_\infty} = \Psi \circ \Phi$$ for unital completely positive maps $\Phi : \ell^2_\infty \oplus_1 \ell^2_\infty \to \mathfrak C_I$ and $\Psi : \mathfrak C_I \to \ell^2_\infty \oplus_1 \ell^2_\infty$. Assume to the contrary that $\mathfrak C_I \otimes_{\rm c} \mathfrak C_I = \mathfrak C_I \otimes_{\max} \mathfrak C_I$. Then, we have completely positive maps
$$\xymatrix{(\ell^2_\infty \oplus_1 \ell^2_\infty) \otimes_c (\ell^2_\infty \oplus_1 \ell^2_\infty) \ar[r]^-{\Phi \otimes \Phi} & \mathfrak C_I \otimes_{\rm c} \mathfrak C_I = \mathfrak C_I \otimes_{\max} \mathfrak C_I \ar[r]^-{\Psi \otimes \Psi} & (\ell^2_\infty \oplus_1 \ell^2_\infty) \otimes_{\max} (\ell^2_\infty \oplus_1 \ell^2_\infty)}.$$
Since the positive cone of the maximal tensor product at each matrix level is the subcone of that of the commuting tensor product, we have $$(\ell^2_\infty \oplus_1 \ell^2_\infty) \otimes_{\rm c} (\ell^2_\infty \oplus_1 \ell^2_\infty) = (\ell^2_\infty \oplus_1 \ell^2_\infty) \otimes_{\max} (\ell^2_\infty \oplus_1 \ell^2_\infty).$$

This contradicts $${\rm NC}(2) \otimes_{\rm c} {\rm NC}(2) \ne {\rm NC}(2) \otimes_{\max} {\rm NC}(2),$$ which was shown in \cite[Corollary 7.12]{FKPT}. Here, ${\rm NC}(n)$ is defined as the operator subsystem ${\rm span} \{ 1, h_1, \cdots, h_n\}$ of the universal $C^*$-algebra generated by self-adjoint contractions $h_1, \cdots, h_n$ as in Definition 6.1 of the same paper. It is unitally completely order isomorphic to the coproduct (involving $n$ terms) $$\ell_\infty^2 \oplus_1 \cdots \oplus_1 \ell_\infty^2.$$

(ii) By \cite[Theorem 5.14]{Ka2}, Kirchberg's conjecture has an affirmative answer if and only if $\ell^2_\infty \oplus_1 \ell^3_\infty$ has the double commutant expectation property. By Theorem \ref{DCEP}, this is equivalent to $(\ell^2_\infty \oplus_1 \ell^3_\infty) \otimes_{\min} \mathfrak C_I = (\ell^2_\infty \oplus_1 \ell^3_\infty) \otimes_c \mathfrak C_I$. By \cite[Theorem 5.9]{Ka2}, this is equivalent to $\mathfrak C_I$ having the double commutant expectation property, and another application of Theorem \ref{DCEP} gives the equivalence with $\mathfrak C_I \otimes_{\min} \mathfrak C_I =\mathfrak C_I \otimes_c \mathfrak C_I$.
\end{proof}

We say that an operator subsystem $\mathcal S$ of a unital $C^*$-algebra $\mathcal A$ {\em contains
enough unitaries} if the unitaries in $\mathcal S$ generate $\mathcal A$ as a $C^*$-algebra. If $\mathcal S \subset \mathcal A$ contains enough unitaries and $\mathcal S \otimes_{\min} \mathcal B \hookrightarrow \mathcal A \otimes_{\max} \mathcal B$ completely order isomorphically for a unital $C^*$-algebra $\mathcal B$, then we have $\mathcal A \otimes_{\min} \mathcal B = \mathcal A \otimes_{\max} \mathcal B$ \cite[Proposition 9.5]{KPTT2}.

\begin{cor}[Kirchberg]\label{Kirchberg}
Let $\mathbb F_{\infty}$ be a free group on a countably infinite number of
generators and $H$ be a Hilbert space. We have $$C^*(\mathbb F_\infty) \otimes_{\min} B(H) = C^*(\mathbb F_\infty) \otimes_{\max} B(H).$$
\end{cor}

\begin{proof}
Since the identity map on $\ell^2_\infty \oplus_1 \ell^3_\infty$ is factorized through $\mathfrak C_I$ by unital completely positive maps, Theorem \ref{main} immediately implies that $$(\ell^2_\infty \oplus_1 \ell^3_\infty) \otimes_{\min} B(H) = (\ell^2_\infty \oplus_1 \ell^3_\infty) \otimes_{\max} B(H).$$ Alternatively, applying the proof of Theorem \ref{main} to commutative algebras instead of matrix algebras, we obtain $$\xymatrix{\ell^5_\infty \otimes_{\min} B(H) \ar@{=}[r]\ar@{->>}[d] & \ell^5_\infty \otimes_{\max} B(H) \ar@{->>}[d] \\ (\ell^2_\infty \oplus_1 \ell^3_\infty) \otimes_{\min} B(H)  & (\ell^2_\infty \oplus_1 \ell^3_\infty) \otimes_{\max} B(H).}$$ For the remaining proof, we follow Kavruk's paper \cite{Ka2}. Since $$\ell^2_\infty \oplus_1 \ell^3_\infty \subset C^*(\mathbb Z_2 \ast \mathbb Z_3)$$ contains enough unitaries \cite[Theorem 4.8]{Ka2}, we have $$C^*(\mathbb Z_2 \ast \mathbb Z_3) \otimes_{\min} B(H) = C^*(\mathbb Z_2 \ast \mathbb Z_3) \otimes_{\max} B(H)$$ by \cite[Proposition 9.5]{KPTT2}. The free group $\mathbb F_\infty$ embeds into the free product $\mathbb Z_2 \ast \mathbb Z_3$ \cite{LH}. By \cite[Proposition 8.8]{P3}, $C^*(\mathbb F_\infty)$ is a $C^*$-subalgebra of $C^*(\mathbb Z_2 \ast \mathbb Z_3)$ complemented by a unital completely positive map.
\end{proof}

A wide class of operator systems shares the properties of $\mathfrak C_I$. Let $\mathcal M = \{ \mathcal M_k \}_{k \in \mathbb N}$ be a sequence of direct sums of matrix algebras such that $$\limsup_{k \to \infty} s(k) = \infty$$ when $\mathcal M_k = M_{d_1} \oplus \cdots \oplus M_{d_n}$ and $s(k) = \max \{d_1, \cdots, d_n \}$. Let $\mathcal M_{\iota_k}$ denote the copy of $\mathcal M_k$ for each index $\iota_k \in I_k$. We define the operator system $\mathfrak C_I (\mathcal M)$ (respectively $\mathfrak C_1 (\mathcal M)$) as the coproduct $$\oplus_1 \{ \mathcal M_{\iota_k} : k \in \mathbb N, \iota_k \in I_k \} \qquad (\text{respectively}~ \oplus_1 \{ \mathcal M_k : k \in \mathbb N \}).$$
In particular, we have $\mathfrak C_I = \mathfrak C_I (\mathcal M)$ and $\mathfrak C_1 = \mathfrak C_1 (\mathcal M)$ when $\mathcal M_k = M_k \oplus M_k$.

\begin{thm}
Suppose that $\mathcal S$ is an operator system. Then,
\begin{enumerate}
\item[(i)] if $\mathcal S^+_{\|\cdot\| \le 1}$ is indexed by a set $I$, then $\mathcal S$ is an operator system quotient of $\mathfrak C_I(\mathcal M)$;
\item[(ii)] if $\mathcal S$ is a countable union of its finite dimensional subsystems, then $\mathcal S$ is an operator system quotient of $\mathfrak C_1 (\mathcal M)$;
\item[(iii)] $\mathfrak C_I(\mathcal M)$ satisfies the operator system lifting property;
\item[(iv)] $\mathfrak C_I (\mathcal M) \otimes_{\min} B(H) = \mathfrak C_I (\mathcal M) \otimes_{\max} B(H);$
\item[(v)] $\mathcal S$ has the double commutant expectation property if and only if $\mathcal S \otimes_{\min} \mathfrak C_I (\mathcal M) = \mathcal S \otimes_{\rm c} \mathfrak C_I (\mathcal M)$;
\item[(vi)] the Kirchberg's conjecture has an affirmative answer if and only if $\mathfrak C_I (\mathcal M) \otimes_{\min} \mathfrak C_I (\mathcal M) = \mathfrak C_I (\mathcal M) \otimes_c \mathfrak C_I (\mathcal M)$.
\end{enumerate}
\end{thm}

\begin{proof}
(i) We take a subsequence $\{\mathcal M_{n_k}\}_{k \in \mathbb N}$ so that $s(n_k) \ge 2k$ and put $\mathcal N_k = \mathcal M_{n_k}$. By Proposition \ref{coproduct} (ii), $\mathfrak C_I (\mathcal N)$ is an operator subsystem of $\mathfrak C_I (\mathcal M)$. Take a state $\omega_l$ on $\mathcal M_l$ for each $l \ne n_k$. By the universal property of the coproduct, there exists a unital completely positive map  $P : \mathfrak C_I (\mathcal M) \to \mathfrak C_I (\mathcal N)$ such that
$$P(x) = \begin{cases} x & \text{if}~ x \in \mathcal M_{n_k} \\ \omega_l (x) 1 & \text{if}~ x \in \mathcal M_l, l \ne n_k. \end{cases}$$
Since $P$ is a unital completely positive projection, $\mathfrak C_I(\mathcal N)$ is an operator system quotient of $\mathfrak C_I(\mathcal M)$.

We may assume that $s(k) \ge 2k$ for each $k \in \mathbb N$. We write $\mathcal M_k = M_{d_1} \oplus \cdots \oplus M_{d_m}$ and $d_l \ge 2k$. The identity map on $M_k \oplus M_k$ is factorized as ${\rm id}_{M_k \oplus M_k} = Q_k \circ J_k$ for the unital completely positive maps
$$J_k : A \in M_k \oplus M_k \mapsto (\omega (A)I_{d_1 + \cdots + d_{l-1}}) \oplus A \oplus (\omega (A)I_{(d_l - 2k) + d_{l+1} + \cdots +d_m}) \in \mathcal M_k$$
(where $\omega$ is a state on $M_k \oplus M_k$) and $$Q_k : A_1 \oplus \cdots \oplus A_m \in \mathcal M_k \mapsto [(A_l)_{i,j}]_{1 \le i,j \le k} \oplus [(A_l)_{i+k,j+k}]_{1 \le i,j \le k} \in M_k \oplus M_k.$$
Let $J : \mathfrak C_I \to \mathfrak C_I (\mathcal M)$ (respectively $Q : \mathfrak C_I (\mathcal M) \to \mathfrak C_I$) be the unital completely positive extension of $J_{\iota_k} : (M_k \oplus M_k)_{\iota_k} \to \mathcal M_{\iota_k}$ (respectively $Q_{\iota_k} : \mathcal M_{\iota_k} \to (M_k \oplus M_k)_{\iota_k}$) over $k \in \mathbb N, \iota_k \in I_k$. Then, the identity map on $\mathfrak C_I$ is factorized as ${\rm id}_{\mathfrak C_I} = Q \circ J$. Hence, $\mathfrak C_I$ is an operator system quotient of $\mathfrak C_I(\mathcal M)$.

(ii) By Theorem \ref{N}, $\mathcal S$ is an operator system quotient of $\mathfrak C_1$. The remaining proof is similar to (i).

(iii), (iv) The proofs of Theorem \ref{LP} and Theorem \ref{main} work generally for coproducts of direct sums of matrix algebras.

(v), (vi) The proofs of Theorem \ref{DCEP} and Corollary \ref{CC} work generally for coproducts of direct sums of matrix algebras which the identity map on $\ell^\infty_2 \oplus_1 \ell^\infty_3$ factorizes through.
\end{proof}

\section{Liftings of completely positive maps}

It is natural to ask whether the universal operator system $\mathfrak C_I$ is a projective object in the category of operator systems. In other words, for any operator system $\mathcal S$ and its kernel $\mathcal J$, does every unital completely positive map $\varphi : \mathfrak C_I \to \mathcal S \slash \mathcal J$ lift to a unital completely positive map $\tilde{\varphi} : \mathfrak C_I \to \mathcal S$? The answer is negative in an extreme manner.

\begin{prop}
An operator system $\mathcal S$ is one-dimensional if and only if for any operator system $\mathcal T$ and its kernel $\mathcal J$, every unital completely positive map $\varphi : \mathcal S \to \mathcal T \slash \mathcal J$ lifts to a completely positive map $\tilde{\varphi} : \mathcal S \to \mathcal T$.
\end{prop}

\begin{proof}
Let $V^+$ be the cone in $\mathbb R^3$ generated by $$\{ (x,y,1) : (x-1)^2 + y^2 \le 1, y \ge 0 \}$$ and the origin. The triple $V:=(\mathbb C^3, V^+, (1,1,2))$ is an Archimedean ordered $*$-vector space. The positive cones of the operator system ${\rm OMAX} (V)$ introduced in \cite{PTT} are given as $$M_n({\rm OMAX} (V))^+ = \{ X \in M_n(V) : \forall \varepsilon > 0, X+\varepsilon I_n \otimes 1_V \in M_n^+ \otimes V^+ \},$$ where $$M_n^+ \otimes V^+ = \{ \sum_{i=1}^m \alpha_i \otimes v_i \in M_n \otimes V : m \in \mathbb N, \alpha_i \in M_n^+, v_i \in V^+ \}.$$

Let $$P : (x,y,z) \in {\rm OMAX} (V) \mapsto (x,y) \in \ell_\infty^2$$ be the projection. We take $\varepsilon>0$ and an element $$\alpha_1 \otimes (1,0) + \alpha_2 \otimes (0,1) \in M_n(\ell^2_\infty)^+ = M_n^+ \oplus M_n^+$$ for nonzero $\alpha_2$. Since $$\begin{aligned} & \alpha_1 \otimes (1,0) + \alpha_2 \otimes (0,1) + \varepsilon I_n \otimes (1,1) \\ = & (\alpha_1 + {\varepsilon \over 2}(I_n-{\alpha_2 \over \|\alpha_2\|})) \otimes (1,0) + \alpha_2 \otimes ({\varepsilon \over 2 \|\alpha_2\|},1)+\varepsilon I_n \otimes ({1 \over 2},1) \end{aligned}$$ lifts to a positive element in $M_n({\rm OMAX} (V))$, the projection $P : {\rm OMAX} (V) \to \ell^2_\infty$ is a complete order quotient map.

Suppose that ${\rm dim} \mathcal S  \ge 2$. Let $v$ be a positive element in $\mathcal S$ distinct from the scalar multiple of the identity. Considering $v-\lambda I$ for sufficiently large $\lambda>0$, we may assume that the spectrum of $v$ contains zero. Let $\omega_1$ and $\omega_2$ be states on $\mathcal S$ that extend, respectively, the Dirac measures $$\delta_{\{ 0\}} : \lambda 1 + \mu v \in {\rm span} \{ 1, v \} \mapsto \lambda \in \mathbb C$$ $$\delta_{\{ \|v\|\}} : \lambda 1 + \mu v \in {\rm span} \{ 1, v \} \mapsto \lambda+\mu \|v\| \in \mathbb C.$$ The unital completely positive map $\varphi : \mathcal S \to \ell^2_\infty$ defined by $\varphi = (\omega_1, \omega_2)$ cannot be lifted to a completely positive map, because the fiber of $\varphi(v)=(0,\|v\|)$ does not intersect $V^+$.
\end{proof}

The absence of completely positive liftings in the above proof is essentially due to Archimedeanization of quotients \cite{PT}. In Corollary \ref{rigidity}, we will see that there is also rigidity, even though some perturbation is allowed.

A linear map between normed spaces is called a quotient map if it maps the open unit ball onto the open unit ball. Between Banach spaces, it suffices to show that the image of the open unit ball is dense in the open unit ball \cite[Lemma A.2.1]{ER}. Suppose that $T : E \to F$ is a bounded linear surjection for normed spaces $E$ and $F$. Let $E_0$ be a dense subspace of $E$, and $Q_0 : E_0 \to Q(E_0)$ be the surjective restriction of $Q$ on $E_0$. Then, $Q_0$ is a quotient map if and only if $\overline{\ker Q_0} = \ker Q$ and $Q$ is a quotient map \cite[7.4]{DF}. This is called the quotient lemma.

Thanks to the quotient lemma, we can describe the 1-exactness of operator systems by incomplete tensor products. For an operator system $\mathcal S$ and a unital $C^*$-algebra $\mathcal A$ with its closed ideal $\mathcal I$, we denote the completion of $\mathcal S \otimes_{\min} \mathcal A$ by $\mathcal S \hat{\otimes}_{\min} \mathcal A$ and the closure of $\mathcal S \otimes \mathcal I$ in it by $\mathcal S \bar{\otimes} \mathcal I$. When $${\rm id}_{\mathcal S} \otimes \pi : \mathcal S \hat{\otimes}_{\min} \mathcal A \to \mathcal S \hat{\otimes}_{\min} \mathcal A \slash \mathcal I$$ is a complete order quotient map with its kernel $\mathcal S \bar{\otimes} \mathcal I$ for any $C^*$-algebra $\mathcal A$ and its closed ideal $\mathcal I$, $\mathcal S$ is called {\em 1-exact}.

\begin{prop}\label{incomplete}
Suppose that $\mathcal S$ is an operator system and $\mathcal A$ is a unital $C^*$-algebra with its closed ideal $\mathcal I$. Then the map $${\rm id}_{\mathcal S} \otimes \pi : \mathcal S \hat{\otimes}_{\min} \mathcal A \to \mathcal S \hat{\otimes}_{\min} \mathcal A \slash \mathcal I$$ is a complete order quotient map with its kernel $\mathcal S \bar{\otimes} \mathcal I$ if and only if the map $${\rm id}_{\mathcal S} \otimes \pi : \mathcal S \otimes_{\min} \mathcal A \to \mathcal S \otimes_{\min} \mathcal A \slash \mathcal I$$ is a complete order quotient map. Hence, an operator system $\mathcal S$ is 1-exact if and only if ${\rm id}_{\mathcal S} \otimes \pi : \mathcal S \otimes_{\min} \mathcal A \to \mathcal S \otimes_{\min} \mathcal A \slash \mathcal I$ is a complete order quotient map for any unital $C^*$-algebra $\mathcal A$ and its closed ideal $\mathcal I$.
\end{prop}

\begin{proof}
The operator space quotient and the operator system quotient of $\mathcal S \hat{\otimes}_{\min} \mathcal A$ by $\mathcal S \bar{\otimes} \mathcal I$ are completely isometric \cite[Theorem 5.1]{KPTT2}. Since $\mathcal S \otimes \mathcal I$ is the kernel of ${\rm id}_{\mathcal S} \otimes \pi : \mathcal S \otimes_{\min} \mathcal A \to \mathcal S \otimes_{\min} \mathcal A \slash \mathcal I$, we can also consider both the operator space quotient and the operator system quotient of $\mathcal S \otimes_{\min} \mathcal A$ by $\mathcal S \otimes \mathcal I$. It is easy to check that the operator space quotient (respectively operator system quotient) $(\mathcal S \otimes_{\min} \mathcal A) \slash (\mathcal S \otimes \mathcal I)$ is an operator subspace (respectively operator subsystem) of the operator space quotient (respectively operator system quotient) $(\mathcal S \hat{\otimes}_{\min} \mathcal A) \slash (\mathcal S \bar{\otimes} \mathcal I)$. If $z \in \mathcal S \otimes \mathcal A$ and $z+\mathcal S \bar{\otimes} \mathcal I$ is positive in the operator system quotient $(\mathcal S \hat{\otimes}_{\min} \mathcal A) \slash (\mathcal S \bar{\otimes} \mathcal I)$, then there exists $x \in \mathcal S \bar{\otimes} \mathcal I$ such that $$z+{\varepsilon \over 2} 1_{\mathcal S} \otimes 1_{\mathcal A} + x \in (\mathcal S \hat{\otimes}_{\min} \mathcal A)^+.$$ Take $x_0 \in \mathcal S \otimes I$ with $\|x-x_0\| < \varepsilon \slash 2$. Considering $(x_0 + x_0^*) \slash 2$, we may assume that $x_0$ is self-adjoint. We have $$z+ \varepsilon 1_{\mathcal S} \otimes 1_{\mathcal A} + x_0 \in (\mathcal S \otimes_{\min} \mathcal A)^+,$$ which implies that $z+\mathcal S \otimes \mathcal I$ is positive in $(\mathcal S \otimes_{\min} \mathcal A) \slash (\mathcal S \otimes \mathcal I)$. Hence, \cite[Theorem 5.1]{KPTT2} immediately implies that the operator space quotient and the operator system quotient of $\mathcal S \otimes_{\min} \mathcal A$ by $\mathcal S \otimes \mathcal I$ are completely isometric.

A unital linear map between operator systems is completely order isomorphic if and only if it is completely isometric \cite[Corollary 5.1.2]{ER}. If a linear map between Banach spaces maps the open unit ball into the open unit ball densely, then it is a quotient map \cite[Lemma A.2.1]{ER}. Combining them with the quotient lemma, we have equivalences:
\begin{enumerate}
\item[] the map ${\rm id}_{\mathcal S} \otimes \pi : \mathcal S \hat{\otimes}_{\min} \mathcal A \to \mathcal S \hat{\otimes}_{\min} \mathcal A \slash \mathcal I$ is a complete order quotient map with its kernel $\mathcal S \bar{\otimes} \mathcal I$
\item[$\Leftrightarrow$] the operator system quotient $(\mathcal S \hat{\otimes}_{\min} \mathcal A) \slash (\mathcal S \bar{\otimes} \mathcal I)$ is completely order isomorphic to $\mathcal S \hat{\otimes}_{\min} \mathcal A \slash \mathcal I$
\item[$\Leftrightarrow$] the operator space quotient $(\mathcal S \hat{\otimes}_{\min} \mathcal A) \slash (\mathcal S \bar{\otimes} \mathcal I)$ is completely isometric to $\mathcal S \hat{\otimes}_{\min} \mathcal A \slash \mathcal I$ (\cite[Theorem 5.1]{KPTT2}, \cite[Corollary 5.1.2]{ER})
\item[$\Leftrightarrow$] the map ${\rm id}_{\mathcal S} \otimes \pi : \mathcal S \hat{\otimes}_{\min} \mathcal A \to \mathcal S \hat{\otimes}_{\min} \mathcal A \slash \mathcal I$ is a complete quotient map with its kernel $\mathcal S \bar{\otimes} \mathcal I$
\item[$\Leftrightarrow$] the map ${\rm id}_{\mathcal S} \otimes \pi : \mathcal S \otimes_{\min} \mathcal A \to \mathcal S \otimes_{\min} \mathcal A \slash \mathcal I$ is a complete quotient map (Quotient Lemma, \cite[Lemma A.2.1]{ER})
\item[$\Leftrightarrow$] the operator space quotient $(\mathcal S \otimes_{\min} \mathcal A) \slash (\mathcal S \otimes \mathcal I)$ is completely isometric to $\mathcal S \otimes_{\min} \mathcal A \slash \mathcal I$
\item[$\Leftrightarrow$] the operator system quotient $(\mathcal S \otimes_{\min} \mathcal A) \slash (\mathcal S \otimes \mathcal I)$ is completely order isomorphic to $\mathcal S \otimes_{\min} \mathcal A \slash \mathcal I$ (\cite[Theorem 5.1]{KPTT2}, \cite[Corollary 5.1.2]{ER})
\item[$\Leftrightarrow$] the map ${\rm id}_{\mathcal S} \otimes \pi : \mathcal S \otimes_{\min} \mathcal A \to \mathcal S \otimes_{\min} \mathcal A \slash \mathcal I$ is a complete order quotient map.
\end{enumerate}
\end{proof}

As pointed out in \cite[Section 5]{KPTT2}, the framework of short exact sequences
$$0 \to \mathcal S \bar{\otimes} \mathcal I \to \mathcal S \hat{\otimes}_{\min} \mathcal A \to \mathcal S \hat{\otimes}_{\min} \mathcal A \slash \mathcal I \to 0$$ with complete tensor products
is inappropriate if  we replace ideals in $C^*$-algebras and $C^*$-quotients by kernels in operator systems and operator system quotients. Even a one-dimensional operator system does not satisfy such exactness. Instead of short exact sequences with complete tensor products, we make a replacement in $${\rm id}_{\mathcal S} \otimes \pi : \mathcal S \otimes_{\min} \mathcal A \to \mathcal S \otimes_{\min} \mathcal A \slash \mathcal I$$ with incomplete tensor products.

\begin{thm}\label{nuclear}
Let $\mathcal S$ be an operator system. Then, the following are equivalent:
\begin{enumerate}
\item[(i)] $\mathcal S$ is nuclear;
\item[(ii)] if $\Phi : \mathcal T_1 \to \mathcal T_2$ is a complete order quotient map for operator systems $\mathcal T_1$ and $\mathcal T_2$, then $${\rm id}_{\mathcal S} \otimes \Phi : \mathcal S \otimes_{\min} \mathcal T_1 \to \mathcal S \otimes_{\min} \mathcal T_2$$ is a complete order quotient map;
\item[(iii)] if $\Phi : \mathfrak C_I \to \mathcal T$ is a complete order quotient map for an operator system $\mathcal T$, then $${\rm id}_{\mathcal S} \otimes \Phi : \mathcal S \otimes_{\min} \mathfrak C_I \to \mathcal S \otimes_{\min} \mathcal T$$ is a complete order quotient map;
\item[(iv)] if $\Phi : {\mathcal T} \to E$ is a complete order quotient map for an operator system $\mathcal T$ and a finite dimensional operator system $E$, then $${\rm id}_{\mathcal S} \otimes \Phi : \mathcal S \otimes_{\min} {\mathcal T} \to \mathcal S \otimes_{\min} E$$ is a complete order quotient map.
\end{enumerate}
\end{thm}

\begin{proof}
(i) $\Rightarrow$ (ii). Maximal tensor products of complete order quotient maps are still complete order quotient maps \cite[Theorem 3.4]{H}. Combining this with the hypothesis, we have a complete order quotient map
$${\rm id}_{\mathcal S} \otimes \Phi : \mathcal S \otimes_{\min} \mathcal T_1 = \mathcal S \otimes_{\max} \mathcal T_1 \twoheadrightarrow \mathcal S \otimes_{\max} \mathcal T_2 = \mathcal S \otimes_{\min} \mathcal T_2$$

(ii) $\Rightarrow$ (i). The proof is motivated by \cite[Theorem 14.6.1]{ER}. Taking $\mathcal T_1$ as a unital $C^*$-algebra  and $\mathcal T_2$ as its $C^*$-quotient, we see that $\mathcal S$ is a 1-exact operator system by Proposition \ref{incomplete}. We take a finite dimensional operator subsystem $E$ of $\mathcal S$ and $\varepsilon>0$. Then, $E$ is a 1-exact operator system \cite[Corollary 5.8]{KPTT2}, or equivalently, a 1-exact operator space \cite[Proposition 5.5]{KPTT2}. Let $E \subset B(\ell_2)$ and  $P_n : \ell_2 \to \ell_2^n$ be the projection given by $P_n((\lambda_i)_{i=1}^\infty) = (\lambda_1, \cdots, \lambda_n)$. For sufficiently large $n$, the truncation mapping
$$\varphi : x \in E \to  P_n x P_n \in M_n$$  is injective with $\|\varphi^{-1}\|_{cb} < 1+\varepsilon'$  for $\varepsilon'= \varepsilon \slash (1+2 \dim E)$ by \cite{P1}, \cite[Theorem 14.4.1]{ER}. Note that $\varphi$ is unital completely positive and $\varphi^{-1}$ is unital self-adjoint. By \cite[Corollary B.11]{BO}, there exists a unital completely positive map $\psi : \varphi(E) \to \mathcal S$ with $\|\varphi^{-1}-\psi\|_{cb} \le 2 \varepsilon' \dim E$.  Though \cite[Corollary B.11]{BO} assumes that the range space is a $C^*$-algebra, its proof still works more generally when the range space is an operator system.

By choosing a faithful state $\omega$ on $M_n$, we can regard the dual space $M_n^*$ as an operator system. Since $\omega$ is faithful on any operator subsystem, $(\varphi (E)^*, \omega|_{\varphi (E)})$ is also an operator system. The element $z$ in $\varphi (E)^* \otimes_{\min} \mathcal S$ corresponding to $\psi : \varphi (E) \to \mathcal S$ canonically is positive \cite[Lemma 8.5]{KPTT2}. Since the duals of the complete order embeddings between finite dimensional operator systems are complete order quotient maps \cite[Proposition 1.15]{FP}, the restriction $R : M_n^* \to \varphi (E)^*$ is a complete order quotient map. By the hypothesis, $$R \otimes {\rm id}_{\mathcal S} : M_n^* \otimes_{\min} \mathcal S \to \varphi (E)^* \otimes_{\min} \mathcal S$$ is also a complete order quotient map. There exists a positive lifting $\tilde{z} \in M_n^* \otimes_{\min} \mathcal S$ of $z+\varepsilon' \omega|_{\varphi (E)} \otimes 1_{\mathcal S}$. The completely positive map $\tilde{\psi} : M_n \to \mathcal S$ corresponding to $\tilde{z}$ satisfies $$\|\psi - \tilde{\psi}|_{\varphi (E)}\|_{cb} \le \varepsilon'.$$ By the Arveson extension theorem, $\varphi : E \to M_n$ extends to a unital completely positive map $\tilde{\varphi} : \mathcal S \to M_n$. We thus obtain a diagram
$$\xymatrix{\mathcal S \ar[ddrr]_{\tilde{\varphi}} && E \ar[ll]_{\iota} \ar[rr]^{\iota} \ar[d]_{\varphi} && \mathcal S \\ && \varphi(E) \ar[urr]^{\psi} \ar[d]^{\iota} && \\ && M_n \ar[uurr]_{\tilde{\psi}} &&}$$ where $\iota$ denote inclusions.

It follows that $$\begin{aligned} \|\tilde{\psi} \circ \tilde{\varphi} (x) - x \| & \le \|\tilde{\psi} \circ \varphi (x) - \psi \circ \varphi (x) \|+\|\psi \circ \varphi (x) - \varphi^{-1} \circ \varphi (x) \| \\ & \le \varepsilon' \|x\| + 2 \varepsilon' \dim E \|x\| \\ & = \varepsilon \|x\|. \end{aligned}$$ for all $x \in E$. Considering the directed set
$$ \{(E, \varepsilon) : \text{$E$ is a finite dimensional
operator subsystem of $\mathcal S$}, \varepsilon >0 \}$$ with the
standard partial order, we can take nets of unital completely positive maps $\varphi_\lambda : \mathcal S \to M_{n_\lambda}$ and completely positive maps $\psi'_\lambda : M_{n_\lambda} \to \mathcal S$ such that $\psi'_\lambda \circ \varphi_\lambda$ converges to the map ${\rm id}_{\mathcal S}$ in the point-norm topology.

Since each $\varphi_\lambda$ is unital, $\psi'_\lambda (I_{n_\lambda})$ converges to $1_{\mathcal S}$. Let us choose a state $\omega_\lambda$ on $M_{n_\lambda}$ and set $$\psi_\lambda(A) = {1 \over \| \psi'_\lambda\|} \psi'_\lambda(A) + \omega_\lambda(A) (1_{\mathcal S} - {1 \over \| \psi'_\lambda \|} \psi'_\lambda (I_{n_{\lambda}})).$$ Then $\psi_\lambda : M_{n_\lambda} \to \mathcal S$ is a unital completely positive map such that $\psi_\lambda \circ \varphi_\lambda$ converges to the map ${\rm id}_{\mathcal S}$ in the point-norm topology. By \cite[Corollary 3.2]{HP}, $\mathcal S$ is nuclear.

(ii) $\Rightarrow$ (iii), (ii) $\Rightarrow$ (iv). Trivial.

(iii) $\Rightarrow$ (ii). Choose a positive element $z$ in $\mathcal S \otimes_{\min} \mathcal T_2$ and $\varepsilon>0$. By Theorem \ref{universal}, we can take a complete order quotient map $\Psi : \mathfrak C_I \to \mathcal T_1$. By the assumption, there exists a positive element $\tilde{z}$ in $\mathcal S \otimes_{\min} \mathfrak C_I$ satisfying $({\rm id}_{\mathcal S} \otimes \Phi \circ \Psi) (\tilde{z})=z+\varepsilon 1$. Thus, ${\rm id}_{\mathcal S} \otimes \Psi (\tilde{z})$ is a positive lifting of $z+\varepsilon 1$.

(iv) $\Rightarrow$ (ii). Choose a positive element $z=\sum_{i=1}^n x_i \otimes y_i$ in $\mathcal S \otimes_{\min} \mathcal T_2$. Take $E$ as a finite dimensional operator subsystem of $\mathcal T_2$ generated by $\{ y_i : 1 \le i \le n \}$ and $\mathcal T$ as $\Phi^{-1}(E)$.
\end{proof}

\begin{rem}
The equivalence of (i) and (ii) was already discovered by Kavruk independently. The proof depends on Kavruk's result that is not yet published.
\end{rem}

\begin{cor}\label{rigidity}
Suppose that $E$ is a finite dimensional operator system and $\omega$ is a faithful state on $E$. The following are equivalent:
\begin{enumerate}
\item[(i)] if $\varepsilon>0$ and $\varphi : E \to \mathcal S \slash \mathcal J$ is a completely positive map for an operator system $\mathcal S$ and its kernel $\mathcal J$, then there exists a self-adjoint lifting $\tilde{\varphi} : E \to \mathcal S$ of $\varphi$ such that $\tilde{\varphi}+\varepsilon \omega 1_{\mathcal S}$ is completely positive;
\item[(ii)] $E$ is unitally completely order isomorphic to the direct sum of matrix algebras.
\end{enumerate}
\end{cor}

\begin{proof}
(i) $\Rightarrow$ (ii). Condition (i) can be rephrased to state that $${\rm id}_{E^*} \otimes \pi : E^* \otimes_{\min} \mathcal S \to E^* \otimes_{\min} \mathcal S \slash \mathcal J$$ is a complete order quotient map for any operator system $\mathcal S$ and its kernel $\mathcal J$. Hence, $E^*$ is a finite dimensional nuclear operator system. Every finite dimensional nuclear operator system is unitally completely order isomorphic to the direct sum of matrix algebras \cite[Corollary 3.7]{HP}. Suppose that $E^*$ is completely order isomorphic to $\oplus_{i=1}^n M_{k_i}$ for some $n, k_i \in \mathbb N$. Taking their duals, we see that $E$ is completely order isomorphic to  $\oplus_{i=1}^n M_{k_i}$. Suppose that the isomorphism maps the order unit of $E$ to a matrix $A$ in $\oplus_{i=1}^n M_{k_i}$. Then, $A$ is positive definite. Let $$A=U^* {\rm diag} (\lambda_1, \cdots, \lambda_m) U, \qquad \lambda_i>0,~ m=\sum_{i=1}^n k_i$$ be a diagonalization of $A$. The mapping $$\alpha \in \oplus_{i=1}^n M_{k_i} \mapsto U^* {\rm diag} (\sqrt{\lambda_1}, \cdots, \sqrt{\lambda_m})\ \alpha\ {\rm diag} (\sqrt{\lambda_1}, \cdots, \sqrt{\lambda_m}) U \in \oplus_{i=1}^n M_{k_i}$$ is a complete order isomorphism that maps the identity matrix to $A$.

(ii) $\Rightarrow$ (i) We may assume that $E=\oplus_{i=1}^n M_{k_i}$. Let $A$ be a density matrix of $\omega$ and $\lambda>0$ be its smallest eigenvalue. Suppose that $z \in \oplus_{i=1}^n M_{k_i}(\mathcal S \slash \mathcal J)$ is the direct sum of Choi matrices corresponding to the restrictions of $\varphi$ on each blocks $M_{k_i}$. There exists a lifting $\tilde{z} \in \oplus_{i=1}^n M_{k_i}(\mathcal S)$ of $z$ such that $\tilde{z}+\varepsilon \lambda I_m \otimes 1_{\mathcal S}$ ($m=\sum_{i=1}^n k_i$) is positive. Let $\tilde{\varphi} : E \to \mathcal S$ be a self-adjoint map corresponding to $\tilde{z}$. Then we have $$\tilde{\varphi}+\varepsilon \omega 1_{\mathcal S} = \tilde{\varphi}+\varepsilon {\rm tr}(\ \cdot\ A) 1_{\mathcal S} \ge_{cp} \tilde{\varphi}+\varepsilon \lambda {\rm tr}(\ \cdot\ ) 1_{\mathcal S} \ge_{cp} 0.$$
\end{proof}

In the last statement of Proposition \ref{incomplete}, an operator systems $\mathcal S$ is fixed, and a $C^*$-algebra $\mathcal A$ and its closed ideal $\mathcal I$ are considered to be variables in $${\rm id}_{\mathcal S} \otimes \pi : \mathcal S \otimes_{\min} \mathcal A \to \mathcal S \otimes_{\min} \mathcal A \slash \mathcal I.$$ In the following, we switch their roles. As a result, we give an operator system theoretic proof of the Effros-Haagerup lifting theorem \cite[Theorem 3.2]{EH}.

\begin{thm}
Suppose that $\mathcal A$ is a unital $C^*$-algebra and $\mathcal I$ is its closed ideal. The following are equivalent:
\begin{enumerate}
\item[(i)] ${\rm id}_{\mathcal S} \otimes \pi : \mathcal S \otimes_{\min} \mathcal A \to \mathcal S \otimes_{\min} \mathcal A \slash \mathcal I$ is a complete order quotient map for any operator system $\mathcal S$;
\item[(ii)] ${\rm id}_{\mathcal B} \otimes \pi : \mathcal B \otimes_{\min} \mathcal A \to \mathcal B \otimes_{\min} \mathcal A \slash \mathcal I$ is a complete order quotient map for any unital $C^*$-algebra $\mathcal B$;
\item[(iii)] ${\rm id}_{B(H)} \otimes \pi : B(H) \otimes_{\min} \mathcal A \to B(H) \otimes_{\min} \mathcal A \slash \mathcal I$ is a complete order quotient map for a separable Hilbert space $H$;
\item[(iv)] ${\rm id}_E \otimes \pi : E \otimes_{\min} \mathcal A \to E \otimes_{\min} \mathcal A \slash \mathcal I$ is a complete order quotient map for any finite dimensional operator system $E$;
\item[(v)] the sequence $$0 \to \mathcal B \otimes_{\rm C^*\min} \mathcal I \to \mathcal B \otimes_{\rm C^*\min} \mathcal A \to \mathcal B \otimes_{\rm C^*\min} \mathcal A \slash \mathcal I \to 0$$ is exact for any $C^*$-algebra $\mathcal B$;
\item[(vi)] for any finite dimensional operator system $E$, every completely positive map $\varphi : E \to \mathcal A \slash \mathcal I$ lifts to a completely positive map $\tilde{\varphi} : E \to \mathcal A$;
\item[(vii)] for any finite dimensional operator system $E$, every unital completely positive map $\varphi : E \to \mathcal A \slash \mathcal I$ lifts to a unital completely positive map $\tilde{\varphi} : E \to \mathcal A$;
\item[(viii)] for any index set $I$, every unital completely positive finite rank map $\varphi : \mathfrak C_I \to \mathcal A \slash \mathcal I$ lifts to a unital completely positive map $\tilde{\varphi} : \mathfrak C_I \to \mathcal A$ with ${\rm Ker} \varphi = {\rm Ker} \tilde{\varphi}$;
\item[(ix)] every unital completely positive finite rank map $\varphi : \mathfrak C_1 \to \mathcal A \slash \mathcal I$ lifts to a unital completely positive map $\tilde{\varphi} : \mathfrak C_1 \to \mathcal A$ with ${\rm Ker} \varphi = {\rm Ker} \tilde{\varphi}$.
\end{enumerate}
\end{thm}

\begin{proof}
(i) $\Rightarrow$ (ii) $\Rightarrow$ (iii) and (viii) $\Rightarrow$ (ix) are trivial. (vi) $\Rightarrow$ (vii) follows from \cite[Remark 8.3]{KPTT2}. (ii) $\Leftrightarrow$ (v) follows from Proposition \ref{incomplete}. For (iii) $\Rightarrow$ (iv) and (iv) $\Rightarrow$ (i), it is sufficient to consider the first matrix level.

(iii) $\Rightarrow$ (iv). Let $E \subset B(H)$ for a separable Hilbert space $H$. Take a strictly positive element $z$ in $E \otimes_{\min} {\mathcal A \slash \mathcal I}$ which is an operator subsystem of $B(H) \otimes_{\min} \mathcal A \slash \mathcal I$. By the assumption, there exists a positive lifting $\tilde{z}$ in $B(H) \otimes_{\min} \mathcal A$. Let $\{ x_i : 1 \le i \le k \}$ be a self-adjoint basis of $E$ and $\{ \widehat{x}_i : 1 \le i \le k \}$ be its dual basis. Each functional $\widehat{x}_i$ on $E$ extends to a continuous self-adjoint functional on $B(H)$ which we still denote by $\widehat{x}_i$. The map $P := \sum_{i=1}^k \widehat{x}_i \otimes x_i : B(H) \to B(H)$ is a self-adjoint projection onto $E$. Since $$({\rm id}_{B(H)}-P) \otimes \pi (\tilde{z}) = z - (P \otimes {\rm id}_{\mathcal A \slash \mathcal I})(z)=0,$$ we have $$({\rm id}_{B(H)}-P) \otimes {\rm id}_{\mathcal A} (\tilde{z}) \in B(H) \otimes \mathcal I.$$ We write $$({\rm id}_{B(H)}-P) \otimes {\rm id}_{\mathcal A}(\tilde{z}) = \sum_{i=1}^n b_i \otimes h_i, \qquad b_i \in B(H)_{sa}, h_i \in \mathcal I_{sa}.$$ Each $h_i$ is decomposed into $h_i=h_i^+ - h_i^-$ for $h_i^+, h_i^- \in \mathcal I^+$. From $$\begin{aligned} 0 \le \tilde{z} & = (P \otimes {\rm id}_{\mathcal A})(\tilde{z}) + \sum_{i=1}^n b_i \otimes h_i^+ - \sum_{i=1}^n b_i \otimes h_i^- \\ & \le (P \otimes {\rm id}_{\mathcal A})(\tilde{z}) + \sum_{i=1}^n \|b_i\| 1 \otimes h_i^+ + \sum_{i=1}^n \|b_i\| 1 \otimes h_i^- \end{aligned}$$ and $$({\rm id}_{B(H)} \otimes \pi)((P \otimes {\rm id}_{\mathcal A})(\tilde{z}) + \sum_{i=1}^n \|b_i\| 1 \otimes h_i^+ + \sum_{i=1}^n \|b_i\| \otimes h_i^-) = z,$$ we see that $$(P \otimes {\rm id}_{\mathcal A})(\tilde{z}) + \sum_{i=1}^n \|b_i\| 1 \otimes h_i^+ + \sum_{i=1}^n \|b_i\| 1 \otimes h_i^- \in E \otimes_{\min} \mathcal A$$ is a positive lifting of $z$.

(iv) $\Rightarrow$ (i). Take a positive element $z=\sum_{i=1}^n x_i \otimes y_i$ in $\mathcal S \otimes_{\min} \mathcal A \slash \mathcal I$. Let $E$ be a finite dimensional operator system generated by $\{ x_i : 1 \le i \le n \}$. Since $E \otimes_{\min} {\mathcal A \slash \mathcal I}$ is an operator subsystem of $\mathcal S \otimes_{\min} \mathcal A \slash \mathcal I$, we have $z$ also positive in $E \otimes_{\min} \mathcal A \slash \mathcal I$. By the hypothesis, there exists a positive element $\tilde{z}$ in $E \otimes_{\min} \mathcal A$ such that $({\rm id}_E \otimes \pi)(\tilde{z})=z$. This element is also positive in $\mathcal S \otimes_{\min} \mathcal A$.

(iv) $\Leftrightarrow$ (vi). Suppose that $E$ is a finite dimensional operator system and $\varphi : E \to {\mathcal A \slash \mathcal I}$ is a completely positive map. The element $z$ in $E^* \otimes_{\min} \mathcal A \slash \mathcal I$ corresponding to $\varphi$ is positive. Since $E$ is finite dimensional, we have $E^* \otimes_{\min} \mathcal A=E^* \hat{\otimes}_{\min} \mathcal A$. The kernel $E^* \otimes \mathcal I$ of ${\rm id}_{E^*} \otimes \pi$ is completely order proximinal in $E^* \otimes_{\min} \mathcal A$ \cite[Corollary 5.1.5]{KPTT2}. By the hypothesis, $z$ lifts to a positive element $\tilde{z}$ in $E^* \otimes_{\min} \mathcal A$. The map $\tilde{\varphi} : E \to \mathcal A$ corresponding to $\tilde{z}$ is completely positive. The converse is merely the reverse of the argument.

(vii) $\Rightarrow$ (vi). The inclusion $\iota :  \varphi (E) + \mathbb C 1_{\mathcal A \slash \mathcal I} \subset \mathcal A \slash \mathcal I$ lifts to a unital completely positive map $\tilde{\iota} : \varphi (E) + \mathbb C 1_{\mathcal A \slash \mathcal I} \to \mathcal A$. The map $\tilde{\iota} \circ \varphi$ is the completely positive lifting of $\varphi$.

(vii) $\Rightarrow$ (viii). Let $Q : \mathfrak C_I \to \mathfrak C_I \slash {\rm Ker} \varphi$ be a quotient map. We have a factorization $\varphi = \psi \circ Q$ for $\psi : \mathfrak C_I \slash {\rm Ker} \varphi \to \mathcal A \slash \mathcal I$. By the hypothesis, $\psi$ lifts to a unital completely positive map $\tilde{\psi} : \mathfrak C_I \slash {\rm Ker} \varphi \to \mathcal A$. Then $\tilde{\psi} \circ Q$ is a unital completely positive lifting of $\varphi$ and their kernels coincide.

(ix) $\Rightarrow$ (vii).  By Theorem \ref{N}, there exists a complete order quotient map $\Phi : \mathfrak C_1 \to E$. The map $\varphi \circ \Phi : \mathfrak C_1 \to \mathcal A \slash \mathcal I$ lifts to a unital completely positive map $\psi : \mathfrak C_1 \to \mathcal A$ such that their kernels coincide. Since ${\rm Ker} \Phi \subset {\rm Ker} \psi$, we get that $\psi$ induces a map $\tilde{\varphi} : E \to \mathcal A \slash \mathcal I$ which is a unital completely positive lifting of $\varphi$.
\end{proof}

The following theorem can be regarded as an operator system version of  the quotient lemma.

\begin{thm}
Suppose that $\Phi : \mathcal S \to \mathcal T$ is a unital completely positive surjection for operator systems $\mathcal S$ and $\mathcal T$. Let $S_0$ be an operator subsystem that is dense in $\mathcal S$, $\mathcal T_0:=\Phi (\mathcal S_0)$, and $\Phi_0 = \Phi |_{\mathcal S_0} : \mathcal S_0 \to \mathcal T_0$ be the surjective restriction. Then, the following are equivalent:
\begin{enumerate}
\item[(i)] $\Phi : \mathcal S \to \mathcal T$ is a complete order quotient map and for any $\varepsilon>0$, $k \in \mathbb N$ and a self-adjoint element $x \in {\rm Ker} \Phi_k$, there exists a self-adjoint element $x_0 \in {\rm Ker} (\Phi_0)_k$ such that $x_0+\varepsilon 1 \ge x$;
\item[(ii)] $\Phi_0 : \mathcal S_0 \to \mathcal T_0$ is a complete order quotient map.
\end{enumerate}
\end{thm}

\begin{proof}
The following arguments apply to all matricial levels.

(i) $\Rightarrow$ (ii). Choose $\varepsilon >0$ and $\Phi_0 (y_0) \in \mathcal T_0^+$ for  a self-adjoint $y_0 \in \mathcal S_0$. By the hypothesis, there exist self-adjoint $x \in {\rm Ker} \Phi$ and $x_0 \in {\rm Ker} \Phi_0$ such that $$y_0+{\varepsilon \over 2}1 +x \in \mathcal S^+ \qquad \text{and} \qquad x \le x_0 +{\varepsilon \over 2}1.$$ It follows that $$y_0+\varepsilon 1 +x_0 \ge y_0+{\varepsilon \over 2}1 +x \ge 0.$$

(ii) $\Rightarrow$ (i). Take $\varepsilon >0$ and a self-adjoint element $x$ in ${\rm Ker} \Phi$. Since $\mathcal S_0$ is dense in $\mathcal S$, there exists a self-adjoint element $y_0$ in $\mathcal S_0$ such that $$x-{\varepsilon \over 3}1 \le y_0 \le x+{\varepsilon \over 3}1,$$ which implies that $$\Phi_0 (-y_0 + {\varepsilon \over 3}1) = \Phi (-y_0 +x+{\varepsilon \over 3}1) \in \mathcal T^+ \cap \mathcal T_0 = \mathcal T_0^+.$$ There exists an element $x_0$ in ${\rm Ker} \Phi_0$ such that $$-y_0 +{2 \over 3} \varepsilon 1+x_0 \ge 0.$$ From $$x-{\varepsilon \over 3}1 \le y_0 \le {2 \over 3} \varepsilon 1 +x_0,$$ it follows that $x \le \varepsilon 1 +x_0$.

Let $\Phi(y) \in \mathcal T^+$ for a self-adjoint $y \in \mathcal S$. There exists an element $y_0$ in $\mathcal S_0$ such that $$y-{\varepsilon \over 3}1 \le y_0 \le y+{\varepsilon \over 3}1,$$ which implies that $$\Phi_0 (y_0 + {\varepsilon \over 3}1) \ge \Phi (y) \ge 0.$$ There exists an element $x_0 \in {\rm Ker} \Phi_0$ such that $$y_0+x_0+{2 \over 3} \varepsilon 1 \ge 0.$$ It follows that $$y+x_0+\varepsilon 1 \ge y_0+x_0+{2 \over 3} \varepsilon 1 \ge 0.$$
\end{proof}

\medskip

\noindent {\bf Acknowledgements.} I would like to thank Ali S. Kavruk and Vern I. Paulsen for their careful reading of the manuscript and helpful comments.

\end{document}